\documentclass[11pt]{article}
\usepackage[margin=1in]{geometry}

\usepackage{preamble}

\newcommand{\indicator}[1]{\mathds{1}\!\!\left\{#1\right\}}
\newcommand{\E}{\mathbb{E}}

\newcommand{\Nzero}{\mathbb{N}_0}
\newcommand{\Prob}{\mathbb{P}}
\newcommand{\Geom}{\mathrm{Geom}}
\newcommand{\Pois}{\mathrm{Pois}}

\newcommand{\iid}{\overset{\text{i.i.d.}}{\sim}}
\newcommand{\dd}{\mathrm{d}}
\newcommand{\equalSpace}{\quad\,\,}
\newcommand{\stoppingRule}{\boldsymbol{r}}
\newcommand{\stoppingTime}{\tau(\boldsymbol{r})}

\newcounter{theorempart}[theorem]

\Crefname{theorempart}{Theorem}{Theorems}

\title{Sample-Based Prophet Inequalities for Random Walks%
    \thanks{Accepted at the 22nd Conference on Web and Internet Economics (WINE 2026).}}
\author{Pieter Kleer, Johan van Leeuwaarden and Daan Noordenbos\\\
Department of Econometrics and Operations Research, Tilburg University\\
\texttt{\{p.s.kleer,j.s.h.vanleeuwaarden,d.noordenbos\}@tilburguniversity.edu}}
\date{\today}

\begin{document}
\maketitle
\thispagestyle{empty}
\begin{abstract}
        We study prophet inequalities for a random walk reward stopping problem with sample-based information. The goal is to stop as close as possible to the maximum of a random walk with i.i.d.~increments, where the performance of a stopping rule is measured as the ratio between the expected reward when stopping and the expected true maximum of the random walk. The latter is the reward of a prophet who gets to see the full random walk beforehand and can stop at the maximum.
         We consider a sample-based model in which the increment distribution of the random walk is unknown and the decision maker is given access to $K$ independent sample paths of the random walk reward process.
        
        For the infinite-horizon setting, we establish a sharp prophet inequality with constant 
        $$
        \left(\frac{K}{K+1}\right)^{K+1}.
        $$
        The guarantee is attained by a randomised stopping rule based on the ladder height decomposition of random walks. 
        As $K\to\infty$, this result recovers the classical $1/e$ prophet inequality from the full information setting (implicitly) known in the literature.
        
        For the finite-horizon setting, where the random walk process terminates after $n$ steps, we first prove a tight no-information prophet inequality with constant $1/H_n$, where $H_n$ denotes the $n$-th harmonic number.
        Next, for $K\ge1$ samples we show that a prophet constant of $1/4$ is attainable. Finally, our most technically challenging result establishes that, with $K$ samples, the prophet constant is at most
        $$
        \left(\frac{K}{K+1}\right)^{K+1}+\frac{6+6H_K}{H_n}
        $$
        for $n\ge 2K^2$, implying convergence to the infinite-horizon constant as $n\to\infty$. Finite horizon random walks are notoriously harder to study, because they lack the Markovian structure that infinite-horizon random walks have. 
        
        Our approach combines random walk theory, including ladder heights and Spitzer's identity, with linear programming duality.
        Our results contribute both to the theory of random walk stopping problems and to the development of sample-based prophet inequalities for correlated rewards, which seems to be a largely unexplored area despite the fact that the sample-based framework has received a lot of attention in recent years. To the best of our knowledge, our tight sample-based prophet inequalities are the first whose performance is parameterised \textit{exactly} (and not only up till constants) by the number of available samples. 
    \end{abstract}

\newpage
\pagenumbering{arabic}
\setcounter{page}{1}

\section{Introduction}

Random walk reward processes are a fundamental tool across probability theory, economics, and computer science to model applications such as the stock market and other financial-decision making processes \cite{cootner1964random,Fama1965original}. 
A random walk is defined by
$
    S_0=0\text{ and }S_k=\sum_{i=1}^kX_i\text{ for }k\ge 1,
$
where $X_1,\dots,X_n$ are independent and identically distributed increments, and the reward is a function of the value of the random walk. The goal is to stop the random walk at a time where the reward is as large as possible.
Unlike classical stopping problems with independent rewards, the sequence $(S_1,\ldots,S_n)$ forms a (strongly) correlated stochastic process, making stopping and prediction problems considerably more challenging.
In this work, we consider the \emph{partial sums stopping problem} where the reward is
$
S_k^+=\max\{0,S_k\}
$
for stopping at time $k$. This reward function models the fact that the stopper can always opt out when rewards are negative.
Determining optimal stopping times for random walk reward processes with full knowledge of the increment distribution of the $X_i$'s, has a long history dating back to the works of \citet{dubins1967optimal} and \citet{darling1972optimal}. Since then there has been a myriad of works focussing on determining the optimal stopping time for such reward processes both in the discrete and continuous-time setting, the latter being known as L\'evy processes; see, e.g., the overview in \cite{lin2024note}. 

An equally interesting, but less studied, setting is the \textit{prophet inequality} version of the problem
that is akin to the competitive ratio in (theoretical) computer science: measure the performance of a stopping rule against the prophet benchmark, which is the true maximum of the reward process that can be observed by a prophet who has seen the whole reward process up front.
The goal is to come up with a stopping time $\tau$ that yields the best possible expected reward guarantee $\E[S_{\tau}^+]$ when compared multiplicatively to the expected prophet benchmark $\E[\max_k S_k^+]$, i.e., come up with the best constant $C \in (0,1]$ for which
$$
\E[S_{\tau}^+] \geq C \cdot \E[\max_k S_k^+],
$$
where the expectation is with respect to the underlying increment distribution and possible randomization of the stopping time. 
\citet{darling1972optimal} were the first to establish such a prophet inequality with $C=1/e$ for the infinite-horizon setting ($n=\infty$) of the partial sums stopping problem, assuming that the random walk increments have a negative expectation.
In the finite horizon setting, where $n < \infty$, Klass \cite{klass1989maximizing} considered the special class of distributions with zero mean showing a prophet inequality of roughly $C=1/2$, and later Wittmann \cite{wittmann1995prophet} gave a prophet inequality with $C = 0.317$ for arbitrary, not necessarily identical, increment distributions. Finally, we remark that Mordecki \cite{mordecki2000optimal,mordecki2002perpetual} derives prophet inequalities for certain L\'evy processes.

The above results typically rely on full distributional knowledge of the underlying increment distribution, which is often deemed unrealistic in practice.
\citet{azar2014prophet} initiated the study of prophet inequality-type problems with limited information, where the stopping time can only be designed based on samples of the involved probability distributions. 
Such information is usually more readily available than the entire distribution.
This paradigm has received a lot of attention in recent years \cite{rubinstein2020optimal, dutting2021prophet, caramanis2022single, correa2022two, gravin2022optimal, kaplan2022online, correa2024sample, cristi2024prophet, fu2024sample, feng2025iid, ezra2026prophet}.
What sets our work apart is that the random walk rewards exhibit a strong dependence, requiring fundamentally different techniques.

In this work we study the prophet inequality version of the partial sums stopping problem with only access to sample information instead of the underlying increment distribution. Specifically, we assume that before the stopping problem begins, the decision-maker is provided with $K$ independent sample paths of the reward process.
The objective is to understand what fraction of the prophet's reward can still be guaranteed uniformly over all increment distributions. The precise model and information structure are introduced in Section~\ref{sec:model}.

\subsection{Our contributions}
Our main prophet inequality results for random walks based on $K$ reward samples paths are given below. In the finite horizon case, we also consider the no-information setting in which zero samples are available (this setting is not relevant in the infinite horizon model).
\begin{itemize}
    \item [(i)] Theorem~\ref{thm:partial-stopping-K-samples} (Infinite-horizon, K samples).
    We establish a sharp prophet constant of $(K/(K+1))^{K+1}$ in the sample-based information model. This is achieved via a randomised stopping rule based on the ladder height decomposition. The constant converges to $1/e$ as $K\to\infty$, matching the classical full-information setting. 
    To the best of our knowledge, this is the first prophet constant that is determined \textit{exactly} as a function of the number of samples. While works such as \citet{correa2019prophet} also study prophet inequalities parameterised by the number of samples, they obtain upper and lower bounds rather than an exact characterisation.
    \item [(ii)] Theorem~\ref{thm:finite-horizon-no-information} (Finite-horizon, no information).
    We prove a tight prophet inequality with constant $1/H_n = \Theta(1/\log n)$, where $H_n$ is the $n$-th harmonic number.
    The bound is attained by an explicit randomised stopping rule derived from Spitzer’s identity.
    \item [(iii)] Theorem~\ref{thm:finite-horizon-impossibility} (Finite-horizon, K samples). 
    We establish the upper bound
    $$
    \left(\frac{K}{K+1}\right)^{K+1}+\frac{6+6H_K}{H_n},\quad\text{(for $n\ge 2K^2$)}
    $$
    on the prophet constant.
    In particular, this bound converges to the infinite-horizon constant as \(n \to \infty\). In Appendix \ref{appendix:finite-horizon-prophet} a stopping rule is given that yields a prophet constant of $1/4$ for $K\ge 1$.
\end{itemize}

These results extend the classical random walk prophet inequality of \citet{darling1972optimal} to a setting with partial information.
More broadly, our results provide progress on the largely unexplored intersection of prophet inequalities for dependent rewards and limited information.

It may appear like Theorem~\ref{thm:finite-horizon-impossibility} immediately implies the optimality of Theorem~\ref{thm:partial-stopping-K-samples} by sending $n\to\infty$. 
This is, however, not quite the full picture. 
Theorem~\ref{thm:finite-horizon-impossibility} relies on an increment distribution that places mass at $-n$. 
In the limit this would place mass at $-\infty$, which is ill-defined.
To transfer the finite-horizon impossibility result to the infinite-horizon setting additional arguments, like Lemma~\ref{lemma:infinite-horizon-samples-finite-reduction}, are required.

\noindent \paragraph{Techniques.}
Our proof methods combine tools from random walk theory with analytical and optimisation based techniques.

A recurring theme is that explicit stopping rules are constructed by exploiting the additional structure provided by random walks. 
In the infinite-horizon setting we use the ladder height decomposition, which represents the maximum as a sum of a geometrically distributed number of i.i.d. ladder height increments. This reduces the stopping problem to a problem involving only geometric random variables. 
In the finite-horizon setting, where the ladder height decomposition is no longer available, we instead rely on Spitzer’s identity, which expresses the expected maximum as a sum of the expected positive partial sums.

A natural stopping rule in the infinite-horizon setting is obtained by estimating the expected number of ladder epochs from the observed samples. Turning this estimate directly into a stopping rule requires rounding, which introduces analytical complications. 
Inspired by \citet[p.87]{romik2015surprising}, we use Poissonisation\footnote{The term Poissonisation does not have a fixed meaning. Our use of it is distinct from, for example, \citet{harb2025new}.} to overcome this obstacle: 
rather than taking a fixed value by rounding the estimate, we take a random value from a Poisson distribution with as mean the (unrounded) estimate.
This smoothing removes discretisation effects and yields a tractable expression for the performance guarantee.

Our impossibility results are obtained via semi-infinite linear programs (LPs). 
By considering a skewed two point increment distribution, upper bounds on the prophet constants can be expressed as linear optimisation problems with infinitely many constraints.
Weak duality then reduces the upper bound analysis to the construction of dual feasible solutions with good objectives. In the no-information setting this yields a short and clean argument. With samples, however, the resulting dual objectives become substantially more involved. We control them through a combination of analytic estimates, probabilistic reinterpretations of relevant sums, and further decompositions.

Our use of LPs goes beyond the typical use of them in stopping problems and online algorithms, such as the framework in \citet{buchbinder2014secretary}, as our LP is infinite-dimensional.
To the best of our knowledge, only a small number of works have recently employed infinite-dimensional linear programming techniques in prophet inequalities, unlike these works, however, our approach yields explicit closed-form constants \citep{perez2025iid, brustle2025splitting}.\\

\noindent It remains a challenge to obtain a matching lower bound, i.e. an algorithmic result with constant $(K/(K+1))^{K+1}$, in the finite-horizon setting.
The main technical obstacle is that the Markovian renewal structure underlying the infinite-horizon setting is lost, requiring an analysis of a more transient flavour. 
In particular, the ladder height decomposition, which is central to our approach, is no longer available. 
This increased difficulty is not unique to the partial information setting, although the presence of samples amplifies it. 
The distinction is already apparent in the full-information setting: the infinite-horizon result of \citet{darling1972optimal} is straightforward, whereas the finite-horizon result of \citet{wittmann1995prophet} requires substantially more care and is not known to be tight. 
In Appendix~\ref{appendix:finite-horizon-prophet} we discuss this obstacle in greater detail and give a $1/4$ prophet inequality for the single sample setting.

\subsection{Related work}\label{sec:related-work}
The classical prophet inequality problem consists of having $n$ sequentially revealed independent random variables $Y_1,\dots,Y_n$. The goal is to stop at the maximum value $Y_{\max} = \max_i Y_i$. The payoff of a stopping rule $\tau$ is the expectation $\E[Y_{\tau}]$ of the chosen value $Y_\tau$, and is compared against $\E[\max_i Y_i]$, which is the payoff a prophet, who can see all outcomes up front, can achieve. Krengel, Sucheston (and Garling)~\cite{krengel1978semiamarts} give a $1/2$-prophet inequality for this problem. \citet{samuel1984comparison} later showed that the same guarantee can be achieved by selecting the first value that exceeds the median of $Y_{\max}$, and \citet{wittmann1995prophet} showed this can also be achieved instead with taking half of the expectation of $Y_{\max}$ (see also \cite{kleinberg2012matroid}). We refer to \citet{lucier2017economic} for extensions of the classical prophet inequality problem.

Despite the extensive growth of the (classic) prophet inequality literature in the last decade, it seems not too much is known about prophet inequalities for stopping partial sums. 
As mentioned earlier, a closely related work to ours is that of \citet{darling1972optimal} who prove a tight prophet inequality of $1/e \approx 0.368$ for the partial sums stopping problem in the infinite-horizon setting when the distribution $F$ is known. 
For the finite-horizon case, the best result is due to \citet{wittmann1995prophet} who gives a prophet inequality of $0.317$ for any finite-horizon $n$ for the partial sums stopping problem. 
In fact, this bound also holds when the increments are only independent, but not necessarily identical. Furthermore, when the increment distributions are i.i.d. with mean zero, \citet{klass1989maximizing} shows a prophet inequality of $1/(2- 1/n)$ where the stopping rule simply constitutes of stopping at the last time $n$. We remark that this inequality is also valid in the no-information setting.

The partial sums stopping problem can also be seen as a prophet inequality problem with correlated rewards. When arbitrary correlations between the $Y_i$ are allowed, \citet{hill1983stop} show a tight prophet inequality of the order $1/n$. \citet{immorlica2023prophet} gave a more fine-grained analysis by introducing a linear correlations model where the $n$ rewards are random variables $Y = (Y_1,\dots,Y_n)$ correlated through a matrix $A$ as $Y=AZ$ where $Z = (Z_1,\dots,Z_n)$ is a vector of independent random variables. They give prophet inequalities in terms of the row and column sparsity of the matrix $A$. A random walk could be fit in this model with the $Z_i$ the increments of the random walk and $A$ a lower-triangular matrix. However, crucially, the framework of \citet{immorlica2023prophet} only works for non-negative random variables $Z_i$ (for which the random walk problem is trivial). Our increment distribution can attain both positive and negative values (and in fact is assumed to have negative mean) to model the fact that a random walk can go up and down. For further recent works related to prophet inequalities with correlations, see \citep{livanos2024improved, caragiannis2021relaxing,kleer2025bayesian}.

Furthermore, relevant for us is a research line on prophet inequalities with limited information, initiated by \citet{azar2014prophet}, where instead of knowing the full distribution the gambler is only given access to samples of the distributions. \citet{rubinstein2020optimal} give a tight prophet inequality of $1/2$ for the classical prophet inequality problem with one sample from each distribution $Y_1,\dots,Y_n$. In recent years, sample-based prophet inequalities for various extensions and variations, such as prophet secretary problems and combinatorial settings, have also been developed \cite{dutting2021prophet,caramanis2022single,correa2022two,gravin2022optimal,kaplan2022online,correa2024sample,cristi2024prophet,fu2024sample,feng2025iid,ezra2026prophet}. Other settings in which full distributional information is not known have also been considered in the literature, such as the works of \citet{dutting2019posted} and \citet{correa2026informativeness} who consider classical prophet inequalities with inaccurate priors and moment-based information respectively.

\section{Model}\label{sec:model}

This section presents the framework for stopping rules with partial information and the random walk model studied in this paper.
The following notation is fixed throughout: $\mathbb{N}_0=\mathbb{N}\cup\{0\}$, $x^+=\max\{0,x\}$, $\log$ is the natural logarithm, and $\mathcal{P}(A)$ denotes the set of all probability measures on the set $A$.

\subsection{Stopping rules with partial information}\label{sec:stopping-rules}
Let $(Z_k)_{k\ge 1}$ be a real-valued stochastic process representing rewards. After observing $Z_1,\dots,Z_k$, a decision-maker must decide whether to stop and select the current reward, or to continue.

The decision-maker is provided with auxiliary information, modelled as a realization $d$ of a $\mathcal{D}$-valued random variable $D$ revealed before any observations of the process. The realisation $d$ may encode partial information about the law of $(Z_k)_{k \geq 0}$, such as an independent sample from its distribution.

A (possibly randomised) \emph{stopping rule} is a family of measurable functions $\{r_k\}_{k\ge 1}$, where $r_k(d,z_1,\dots,z_k)\in[0,1]$ denotes the probability of stopping at the $k$-th value given the auxiliary data $D=d$ and the observed trajectory $Z_1=z_1,\dots, Z_k=z_k$, conditional on not having stopped previously. The set of all such stopping rules is denoted by $\mathcal{R}(\mathcal{D})$. 

Each stopping rule $\stoppingRule\in \mathcal{R}(\mathcal{D})$
induces a stopping time $\stoppingTime\in\mathbb{N}\cup\{\infty\}$ whose conditional distribution satisfies
$$
\Prob(\stoppingTime=k|D=d,Z_1 = z_1,\dots,Z_k=z_k)=r_k(d,z_1,\dots,z_k)\prod_{j=1}^{k-1}(1-r_j(d,z_1,\dots,z_j))
$$
The expected reward under the stopping rule $\stoppingRule$ is denoted by $\E[Z_{\stoppingTime}]$.
\subsection{Random walk stopping with limited information}

We study a stopping problem for random walks under limited information.
Let $n\in\mathbb{N}$ denote the horizon, 
let the increments $X_1,\dots,X_n$ be i.i.d. random variables with distribution $F$, 
and let $(S_k)_{0\le k\le n}$ be the associated random walk defined by
$$
    S_0=0\text{ and }S_k=\sum_{i=1}^kX_i.
$$ 
We consider the \emph{partial sums stopping problem}, where the rewards $S_1^+,\dots,S_n^+$ (recall that $S_k^+ = \max\{0, S_k\}$) are revealed sequentially, and after each reward is revealed, the decision-maker must choose either to accept it and stop, or to discard it and continue. 
The objective is to maximise the expected selected reward relative to the prophet benchmark $M_n = \max_{0\le k\le n} S_k^+$.

The decision-maker does not observe the increment distribution $F$,
but instead has access to $K$ independent sample paths of the reward process,
$$
D=(S_{k,i}^+)_{0\le k\le n, 1\le i\le K}, 
$$
where $S_{k,i}$ denotes the $k$-th partial sum of the $i$-th sample path. We refer to this information structure as the \emph{reward-sample model}.

A stronger information structure is the \emph{increment-sample model}, in which the decision-maker observes the full sample paths $(S_{k,i})_{0\le k\le n, 1\le i\le K}$ (equivalently, the increments $X_{k,i}$). 
Since the reward-sample model only reveals the positive parts of the partial sums, it contains strictly less information and is therefore a more restricted variant of the problem.

The above model extends naturally to the infinite-horizon setting $n=\infty$, but some additional care is required.
To ensure that the maximum $M=\max_{k\ge 0}S_k^+$ is almost surely finite, we impose the standard assumption that $\E[X_1]=\mu<0$.
Under this assumption, the random walk satisfies $S_k \to -\infty$ almost surely, so it is positive only finitely often. Consequently, the reward sample paths admit a finite representation.

By contrast, the increment-sample model becomes problematic when $n=\infty$. An infinite increment sample path does not admit a finite representation, and observing infinitely many increments effectively reveals the distribution $F$, reducing the problem to a full-information setting. The reward-sample model avoids these issues and remains a genuinely limited-information model.

A natural question, however, concerns the implementability of the reward sample model: although each sample path admits an almost surely finite representation,
the amount of time needed to acquire this representation (or even a statistic of it, like for example the amount of ladder epochs) is not uniformly bounded.
In Section \ref{sec:random-horizon}, we address this by introducing geometric killing, whereby the process is terminated independently at each step with a small fixed probability. The resulting model admits uniformly bounded sample representations with high probability while remaining arbitrarily close to the original infinite-horizon problem.

In the next section, we present our main result for the infinite-horizon setting: a tight prophet inequality with constant $(K/(K+1))^{K+1}$.

\section{Infinite-horizon}
To derive our results for the infinite-horizon setting, we first recall key structural properties of random walks and their maxima (see \citet{feller2} for further details).

Central to our analysis is the \emph{ladder height process}, which describes the successive moments at which the random walk attains a new maximum.
Formally, the \emph{strictly ascending ladder times} are defined recursively by $
T_0=0$ and $T_{i+1}=\inf\{k>T_i:S_k>S_{T_i}\}$. 
The corresponding \emph{ladder height increments} are given by 
$J_i=S_{T_{i+1}}-S_{T_{i}}$,
with the associated inter arrival times $I_i=T_{i+1}-T_i$.
The pairs $(I_i,J_i)$ form an i.i.d. sequence.

This viewpoint leads to the \emph{ladder height decomposition}, which expresses the global maximum of the random walk as the sum of its successive record increments,
\begin{equation}
    \label{eq:ladder-height-decomposition}
    M=\sum_{i=1}^{G}J_i,
\end{equation}
where $G=\max\{k:T_k<\infty\}$ denotes the (random) number of ladder epochs. 
Classical results imply that $G$ has a geometric distribution and is independent of the increments $J_i$.
Consequently, using the independence and i.i.d. structure, we obtain $\E[M]=\E[G]\E[J_1]$, which is a direct result of Wald's equation.

The decomposition in \eqref{eq:ladder-height-decomposition} suggests analysing the problem at the level of ladder heights rather than individual increments. This isolates the relevant random walk structure, since the distributional details of the increments enter only through the ladder height distribution. In particular, the stopping problem reduces to one involving geometric random variables.

Specifically, given $K$ samples $G_1,\dots,G_K\iid\Geom(1-p)$, we seek a (possibly randomised) stopping time $Y_{G_1,\dots,G_K}$ such that 
$$
\frac{\E[Y_{G_1,\dots,G_K}\indicator{G\ge Y_{G_1,\dots,G_K}}]}{\E[G]}\ge C_K
$$
for all parameter values $p\in(0,1)$. 
If $p$ were known, then the optimal stopping moment would be $\lfloor1+\E[G]\rfloor$.\footnote{The optimal stopping moment maximises $f_n=\E[n\indicator{G\ge n}]=np^n$. Note that $f_{n+1}/f_n=(1+1/n)p$, so $f_{n+1}\ge f_n\iff n\le p/(1-p)$. Hence, $f_n$ increases up to $n\le p/(1-p)$ and then decreases, so it attains a maximum over the integers at $\lfloor1+p/(1-p) \rfloor=\lfloor1+\E[G]\rfloor$.}
This suggests a stopping rule where $\E[G]$ is replaced with $\bar{G}=\frac{1}{K}\sum_{i=1}^KG_i$. This, however, is not analytically tractable.
To circumvent this, we introduce an additional layer of randomisation.
\begin{definition}[Poissonised stopping rule]
\label{def:pois-s-rule}
Let $\bar G=\frac1K\sum_{i=1}^K G_i$,
where $G_i$ denotes the number of ladder epochs of the $i$-th reward sample path.
Conditional on $\bar G$, let $Y-1\mid \bar G\sim\Pois(\bar G).$
The Poissonised stopping rule stops at the $Y$-th ladder epoch of the target random walk if it exists; otherwise it never stops and gets a reward of zero.
\end{definition}
\noindent The following theorem establishes that the Poissonised stopping rule achieves the optimal guarantee.
\begin{theorem}
    \label{thm:partial-stopping-K-samples} 
    Let $(S_k)_{k\ge 0}$ be a random walk with increments $X_t\iid F$, with $\E[X_1] = \mu < 0$,\footnote{Negative drift ensures that the maximum of the random walk is a.s. finite, and the problem well-defined.} and for $i=1,\dots,K$ let $(S_{t,i})_{t\ge 0,1\le i\le K}$ be independent copies. 
    \begin{itemize}
        \item [(a)] \refstepcounter{theorempart}\label{thm:partial-stopping-K-samples-prophet} (Prophet inequality) 
        Let $\stoppingRule$ denote the Poissonised stopping rule of Definition~\ref{def:pois-s-rule}. Then
        $$
            \E[S_{\stoppingTime}^+]\ge\left(\frac{K}{K+1}\right)^{K+1}\E[\max_{k\ge 0}S_k].
        $$
        \item [(b)] \refstepcounter{theorempart}\label{thm:partial-stopping-K-samples-impossibility} (Sharpness) 
        Among all stopping rules that have access to auxiliary data $D=(S_{t,i}^+)_{t\ge 0,1\le i\le K}$, the constant $(K/(K+1))^{K+1}$ is optimal.
    \end{itemize} 
\end{theorem}

\noindent We first provide a proof sketch for Theorem \ref{thm:partial-stopping-K-samples-impossibility}, followed by the proof of Theorem ~\ref{thm:partial-stopping-K-samples-prophet}. 
The key idea behind Theorem \ref{thm:partial-stopping-K-samples-impossibility} is to construct a hard instance using a two-point increment distribution with highly likely $+1$ steps and rare large negative jumps $-L$ for some large $L$. The deterministic positive increments ensure that the process reveals essentially no additional information over time, while the large negative jumps make recovery to a positive partial sum after a decrement exceedingly unlikely. Consequently, with high probability the process behaves as if it terminates after its first negative jump. This reduction is made precise using tail bounds for the maximum of the associated random walk.

Up to a loss in the upper bound, this reduces the problem to a stopping problem with a geometric horizon. The resulting problem is sufficiently tractable to be reformulated as a semi-infinite LP over stopping times. The desired upper bound is then obtained via weak duality by constructing a carefully chosen dual feasible solution. Finally, a sharp asymptotic analysis shows that both the reduction error and the dual upper bound converge to the tight constant $C_K$. The full proof can be found in Appendix \ref{appendix:infinite-horizon-impossibility-proof}. 

\subsection{Proof of Theorem \ref{thm:partial-stopping-K-samples-prophet}}
    Let $M=\max_{k\ge 0}S_k$. Recall that $M=\sum_{i=1}^{G}J_i$ where $J_i$ are i.i.d. random variables and $G\sim\Geom(1-p)$ is independent of the $J_i$'s and that $\E[M]=\E[G]\E[J_1]$. 
    From each of the $K$ reward sample paths an independent realisation $G_i$ from $\Geom(1-p)$ can be constructed by counting the amount of successive maxima. 
    Now, consider a stopping rule $\stoppingRule$ that stops at the $Y_{G_1,\dots,G_K}$-th ladder height, where $Y_{G_1,\dots,G_K}$ is an integer-valued random variable indexed by $G_1$ to $G_K$.
    The payoff and expected payoff under this stopping rule are
    \begin{align*}
        S_{\stoppingTime}^+&=\indicator{G\ge Y_{G_1,\dots,G_K}}\sum_{i=1}^{Y_{G_1,\dots,G_K}} J_i,\\
            \E[S_{\stoppingTime}^+]&=\E[Y_{G_1,\dots,G_K}\indicator{G\ge Y_{G_1,\dots,G_K}}]\E[J_1],
    \end{align*}
    respectively, because the $J_i's$ are independent of $G$ and $Y_{G_1,\dots,G_K}$. Consequently, 
    \begin{align*}
        \frac{\E[S_{\stoppingTime}^+]}{\E[M]}&=\frac{\E[Y_{G_1,\dots,G_K}\indicator{G\ge Y_{G_1,\dots,G_K}}]\E[J_1]}{\E[G]\E[J_1]}\\&=\frac{\E[Y_{G_1,\dots,G_K}\indicator{G\ge Y_{G_1,\dots,G_K}}]}{\E[G]}:=\mathrm{PC}(K,p).
    \end{align*}
As mentioned previously, if the parameter $p$ would be known, then the optimal stopping moment would be $\lfloor1+\E[G]\rfloor$. 
    Intuitively, a natural candidate is the unbiased estimator 
    $Y_{g_1,\dots,g_K}=\left\lfloor1+\frac{1}{K}\sum_{i=1}^Kg_i\right\rfloor$. However, this particular choice is analytically intractable due to the presence of the floor function.
    Inspired by the Poissonisation trick used in the analysis of the Plancherel measure \cite[p.87]{romik2015surprising}, the randomised stopping moment $Y_{g_1,\dots,g_K}-1\sim\text{Pois}\left(\frac{1}{K}\sum_{i=1}^Kg_i\right)$ is chosen. This makes the analysis more regular by eliminating rounding. 
    
    Note that $\sum_{i=1}^KG_i$ has a negative binomial distribution with parameters $K$ and $1-p$. Using this fact the performance of the above stopping rule can be calculated as
    \begin{align*}
        \mathrm{PC}(K,p)&:=\frac{\E[Y_{G_1,\dots,G_K}\indicator{G\ge Y_{G_1,\dots,G_K}}]}{\E[G]}\\
        &=\frac{1-p}{p}\E[Y_{G_1,\dots,G_K}p^{Y_{G_1,\dots,G_K}}]\\
        &= \frac{1-p}{p}\sum_{n=0}^\infty \mathbb{P}\left[\sum_{i=1}^K G_i = n\right]\cdot \E\left[Y_{G_1,\dots,G_K}p^{Y_{G_1,\dots,G_K}} \ | \ \sum_{i=1}^K G_i = n \right] \\
        &=\frac{1-p}{p}\sum_{n=0}^{\infty}\binom{n+K-1}{K-1}(1-p)^Kp^n\left(\sum_{k=0}^\infty(k+1)p^{k+1}\cdot\frac{\left(n/K\right)^ke^{-n/K}}{k!}\right),
    \end{align*}
    where we use the definition of the probability mass function of the shifted Poisson distribution $1 + \text{Pois}(n/K)$ in the final equality.
    Using the identity $\sum_{k=0}^\infty(k+1)x^k/k!=e^x(1+x)$, with $x = pn/K$, the inner sum can be simplified to
    \begin{align*}
        \sum_{k=0}^\infty(k+1)p^{k+1}\cdot\frac{\left(n/K\right)^ke^{-n/K}}{k!}=pe^{-n/K}\sum_{k=0}^\infty(k+1)\frac{(pn/K)^k}{k!}=pe^{-(1-p)n/K}(1+pn/K).
    \end{align*}
    Substituting this yields
    \begin{align*}
        \mathrm{PC}(K,p)&=\frac{1-p}{p}\sum_{n=0}^{\infty}\binom{n+K-1}{K-1}(1-p)^Kp^n\left(pe^{-(1-p)n/K}(1+pn/K)\right)
        \\
        &=(1-p)^{K+1}\sum_{n=0}^{\infty}\binom{n+K-1}{K-1}\left(pe^{-(1-p)/K}\right)^n(1+pn/K).
    \end{align*}
    Define $q(p)=pe^{-(1-p)/K}$ and note that $q(p) \in (0,1)$.
    With the binomial identities
    $$
    \sum_{n=0}^{\infty}\binom{n+K-1}{K-1}x^n=\frac{1}{(1-x)^K}\text{ and }\sum_{n=0}^{\infty}\binom{n+K-1}{K-1}nx^n=\frac{Kx}{(1-x)^{K+1}}\text{ for }|x|<1
    $$
    the expression with the infinite sums can be simplified to
    \begin{align}
        \label{eq:poisson-rule-performance}
        \mathrm{PC}(K,p)&=(1-p)^{K+1}\left(\sum_{n=0}^{\infty}\binom{n+K-1}{K-1}q(p)^n+\frac{p}{K}\sum_{n=0}^{\infty}\binom{n+K-1}{K-1}nq(p)^n\right)\nonumber\\
        &=(1-p)^{K+1}\left((1-q(p))^{-K}+\frac{p}{K}
        \cdot Kq(p)(1-q(p))^{-(K+1)}\right)\nonumber\\
        &=(1-p)^{K+1}\left((1-q(p))(1-q(p))^{-(K+1)}+pq(p)(1-q(p))^{-(K+1)}\right)\nonumber\\
        &=\left(\frac{1-p}{1-q(p)}\right)^{K+1}(1-q(p)(1-p)).
    \end{align}
    In Lemma \ref{lemma:pc-k-p-infimum} the proof is finished by showing that
    \begin{equation}
        \inf_{p\in(0,1)}\mathrm{PC}(K,p)=\left(\frac{K}{K+1}\right)^{K+1}.
    \end{equation}

\subsection{Random horizon}\label{sec:random-horizon}

A natural question concerns the implementability of the stopping rule from Theorem \ref{thm:partial-stopping-K-samples-prophet}. 
Implementing this rule requires counting the number of ladder epochs in each reward sample path. 
The difficulty is that it is not clear how long one must wait after the last observed ladder height before concluding (with high probability) that no further ladder heights will occur. 
In fact, there is no uniform bound on the tail behaviour of the ascending ladder height times $I_i$, as $\Prob(I_i>n)$ may decay arbitrarily slowly \cite{rogozin1971distribution}. Consequently, even if one only aims for a high-probability guarantee, there is no a priori bound on the amount of data that must be processed in order to count the ladder heights.
An apparent way to sidestep this concern is to consider the finite-horizon setting of Section~\ref{sec:finite-horizon}.
In this setting, however, the ladder height decomposition in \eqref{eq:ladder-height-decomposition} is lost, and we do not recover comparably sharp results.

To obtain a formulation that is both implementable and structurally close to the infinite-horizon setting, we introduce geometric killing: at each time step, the process is independently terminated with a small fixed probability. Geometrically killed random walks are standard in the random walk literature, and they turn infinite-horizon problems into approximately finite-horizon ones while preserving useful structure \cite{kyprianou2006introductory}. One may also view this modification as modelling an external disruption, such as a crisis, that can occur at any time.

Formally, fix $\alpha\in(0,1)$ known to the decision-maker, and let $T,T_1,\dots,T_K \iid \Geom(\alpha)$ be independent termination times. Define the killed reward process by
$Z_k := \indicator{T\ge k} S_k.$
The decision-maker has access to data
$
D = (S_{0,1}^+,\dots,S_{T_1,1}^+,S_{0,2}^+,\dots,S_{T_2,2}^+,\dots,S_{0,K}^+,\dots,S_{T_K,K}^+)
$. 
Under this formulation, the decision-maker may impose a time limit, depending on $K$ and $\alpha$, up to which each process is observed. 
Because of the killing, the terminations times have a known distribution, and the decision-maker can ensure, with high probability, that all (relevant) data has been observed.\footnote{More precisely, the distribution of the total amount of data $|D|$ is known, so for any $\epsilon>0$ one can choose a threshold $L(K,\alpha,\epsilon)$ such that $\Prob(|D|>L(K,\alpha,\epsilon))<\epsilon$. Thus, with a fixed computational budget, the number of ladder heights in the samples can be observed 
with high probability.}
\color{black}

This modification preserves the Markovian structure of the ladder heights, and so the decomposition in \eqref{eq:ladder-height-decomposition} remains valid. 
The only notable change is that the parameter $p$ associated with the amount of ladder heights is now smaller and restricted to $p\in(0,1-\alpha)$ rather than $p\in(0,1)$.

Denote by $\mathrm{PC}_\alpha(K)$ the prophet constant in the killed model. Formally, 
$$
\mathrm{PC}_\alpha(K):=\sup_{\stoppingRule\in\mathcal{R}(\mathcal{D})}\inf_{F\in\mathcal{P}(\mathbb{R})}\frac{\E[Z^+_{\stoppingTime}]}{\E[\max_{k\ge 0}Z_k]}.
$$
The analysis of the Poissonised stopping rule from Theorem~\ref{thm:partial-stopping-K-samples-prophet} carries over directly, with the only modification that $p\in(0,1-\alpha)$. Consequently,
$$
\mathrm{PC}_\alpha(K)\ge \inf_{p\in(0,1-\alpha)}\mathrm{PC}(K,p)\ge \inf_{p\in(0,1)}\mathrm{PC}(K,p)=\left(\frac{K}{K+1}\right)^{K+1}.
$$
In words, the same stopping rule remains valid in the geometric horizon setting.

Moreover, as $\alpha\to 0$, $\mathrm{PC}_\alpha(K)\to (K/(K+1))^{K+1}$.
This can be established by adapting the impossibility argument from Theorem~\ref{thm:partial-stopping-K-samples-impossibility}: the role of rare large negative jumps is now played by termination events, and the reduction leads to a simplified geometric stopping problem. Since this reduction becomes exact as $\alpha\to 0$, the resulting upper bound converges to the lower bound. We leave the details to the interested reader.

\section{Finite-horizon setting}\label{sec:finite-horizon}

In this section we study the partial sums stopping problem in the finite-horizon setting. As noted before, this setting is substantially more difficult than the infinite-horizon case, because the process is no longer Markovian and because the ladder height decomposition in \eqref{eq:ladder-height-decomposition} no longer applies.
This is already apparent in the full-information setting. The infinite-horizon setting was resolved by \cite{darling1972optimal}, while the finite-horizon setting remains open, with the best results due to \cite{wittmann1995prophet}.

Despite this additional difficulty, we can still prove several results. In particular, when no information about the increment distribution is available, we obtain a tight $1/H_n$ prophet inequality.

For $K=1$, and therefore also $K\ge 1$, we give a $1/4$-prophet inequality. Moreover, we prove an impossibility result for the prophet constant that converges to the infinite-horizon bound as $n\to\infty$.

\subsection{No information}
Even without any information about the increment distribution, the finite-horizon setting admits non-trivial results because of the additional structure imposed by the random walk. In particular, we obtain an $O(1/\log n)$ guarantee.

Our algorithmic result is derived from Spitzer's identity, and the matching impossibility result follows from weak duality for semi-infinite LPs. This setting also provides a clean illustration of the duality method used throughout our impossibility proofs.
\begin{theorem}
    \label{thm:finite-horizon-no-information}
    Let $(S_k)_{0\le k\le n}$ be a random walk with increments $X_i\iid F$. Then there is a stopping time $\tau\le n$, that has no knowledge of $F$, such that
    $$
    \E[S_{\tau}^+]\ge\frac{1}{H_n}\E\left[\max_{0\le k\le n}S_k\right],
    $$
    where $H_n$ is the $n$-th harmonic number. Moreover, the constant $1/H_n$ is optimal.
\end{theorem}
\begin{proof}
    The algorithmic result relies on Spitzer's identity, which expresses the expected finite-horizon maximum in terms of the positive parts of the partial sums. Let $M_n=\max_{0\le k\le n}S_k$. Spitzer's identity \text{(see \citet{feller2})} states that 
    \begin{align*}
        \E[M_n] = \sum_{k=1}^n \frac{1}{k}\E[S_k^+].
    \end{align*}
    By dividing by the harmonic number $H_n=\sum_{k=1}^n\frac{1}{k}$ the right-hand side becomes a convex combination of $\E[S_k^+]$'s, and therefore admits an interpretation as the performance of a randomised stopping rule. Let $\tau$, independent of the process, satisfy $\mathbb{P}(\tau=k)=1/(kH_n)$ for $k=1,\dots,n$, then we obtain the following prophet inequality
    $$
    \E[S_{\tau}^+]:=\sum_{k=1}^n\E[S_k^+\indicator{\tau=k}]=\sum_{k=1}^n\mathbb{P}(\tau=k)\E[S_k^+]=\frac{1}{H_n}\sum_{k=1}^n\frac{1}{k}\E[S_k^+]=\frac{1}{H_n}\E[M_n].
    $$
    Using independence in the second step.
    
    Now the impossibility result. 
    Denote by $F_p$ the increment distribution that assigns a probability of $p$ to $+1$ and $1-p$ to $-n$. By restricting to this parametric family, the following upper bound on the prophet constant is obtained:
    $$
    \mathrm{PC}:=\sup_\tau\inf_{F\in\mathcal{P}(\mathbb{R})}\frac{\E[S_\tau^+]}{\E[\max_{0\le k\le n}S_k]}\le\sup_{\tau}\inf_{\{F_p:p\in(0,1)\}}\frac{\E[S_\tau^+]}{\E[\max_{0\le k\le n}S_k]}.
    $$
    For this increment distribution, $S_k^+=k\indicator{M_n\ge k}$
    and $M_n$ is the number of initial $+1$'s, so $M_n=\min\{n,N\}$, where $N\sim\Geom(1-p)$. 
    Because of this the performance of the stopping rule can be written as
    \begin{align}
        \label{eq:tau-write-up}
        \E[S_\tau^+]&=\sum_{m=0}^n\Prob(M_n=m)\sum_{k=0}^mk\Prob(\tau=k|S_1=1,\dots,S_k=k)\nonumber\\
        &=\sum_{k=1}^nk\Prob(M_n\ge k)\Prob(\tau=k|S_1=1,\dots,S_k=k).
    \end{align}
    Note that the probabilities $\Prob(\tau=k|S_1=1,\dots,S_k=k)$ for $k=1,\dots,n$ sum to (less than) one, so the upper bound can be written as
    $$
    \mathrm{PC}\le \sup_{\{q_i\ge 0:q_1+\dots+q_n\le 1\}}\inf_{p\in(0,1)}\frac{\sum_{k=1}^nk\Prob(M_n\ge k)q_k}{\E[M_n]}=\sup_{\{q_i\ge 0:q_1+\dots+q_n\le 1\}}\inf_{p\in(0,1)}\frac{\sum_{k=1}^nkp^kq_k}{\frac{p}{1-p}(1-p^n)},
    $$
    using $\E[M_n]=\frac{p}{1-p}(1-p^n)$ and $\Prob(M_n\ge k)=p^k$ in the last step.
    Because it is optimal for the $q_k$'s to sum to one, the above upper bound can be rewritten as the following semi-infinite LP
    {
    \nobreakdisplays
    \begin{align*}
        &\sup && C\\
        &\text{subject to} &&\frac{1-p}{p(1-p^n)}\sum_{k=1}^nkp^kq_k\ge C  \text{ for all } p\in(0, 1),\\
        &&& C \in \mathbb{R}, \sum_{k=1}^n q_k = 1, \ q_k \geq 0 \text{ for } k = 1,\dots,n.
    \end{align*}
    }
    The dual of the above LP follows from the general form in Appendix \ref{sect:infty-lp-prelim} and is 
    {
    \nobreakdisplays
    \begin{align}
        \label{eq:no-info-dual}
        &\inf && \lambda\\
        &\text{subject to} &&\lambda\ge\int_0^1k\frac{(1-p)p^{k-1}}{1-p^n}\mathrm{d}\mu(p)\text{ for all }k=1,\dots,n,\nonumber\\
        &&& \mu\in\mathcal{P}((0,1))\text{, }\lambda\in\mathbb{R}\nonumber.
    \end{align}
    }
    Note that the measure $\mu$ is the dual variable associated with the constraints for $p \in (0,1)$ and $\lambda$ the dual variable identified with the normalization constraint.

    A feasible dual solution is now constructed.
    Motivated by the heuristic that an optimal solution should equalise the constraints, we choose $\mu$ such that $\int_0^1k\frac{(1-p)p^{k-1}}{1-p^n}\mathrm{d}\mu(p)$ is constant in $k$. This suggests the density $\mathrm{d}\mu(p)=\frac{1}{H_n}\frac{1-p^n}{1-p}\mathrm{d}p$. Note that $\int_0^1\frac{1-p^n}{1-p}\mathrm{d}p=H_n$, so it is a probability measure. Moreover,
    $$
    \int_0^1k\frac{(1-p)p^{k-1}}{1-p^n}\mathrm{d}\mu(p)=\frac{1}{H_n}\int_0^1kp^{k-1}\mathrm{d}p=\frac{1}{H_n},
    $$
    so $\lambda=1/H_n$ and the measure $\mu$ constitute a feasible solution of the dual. 
    By weak duality of semi-infinite LPs, the primal value is at most the value of the dual feasible solution constructed above. As the primal value is an upper bound on the prophet constant, the prophet constant is at most $1/H_n$.
    This matches the performance of the stopping rule considered earlier, establishing optimality of the prophet constant.
\end{proof}

\subsection{\texorpdfstring{$K$}{K} sample finite-horizon impossibility}
By extending the ideas behind Theorem \ref{thm:finite-horizon-no-information}, we can prove the following impossibility result.

\begin{theorem}
    \label{thm:finite-horizon-impossibility} 
    Let $(S_k)_{0\le k\le n}$ be a random walk with increments $X_t\iid F$ and for $i=1,\dots,K$ let $(S_{k,i})_{k\ge 0}$ be independent copies. Consider stopping rules $\stoppingRule = (r_k)_{k \geq 0}$ with $r_k$ based solely on the (realised) rewards $(S_1^+,\dots,S_k^+)$ and sample data $D=(S_{t,i}^+)_{0\le t\le n, 1\le i\le K}$, and define the $K$-sample finite-horizon prophet constant 
    $$
    C_{n,K}:=\sup_{\stoppingRule\in\mathcal{R}(\mathcal{D})}\inf_{F\in\mathcal{P}(\mathbb{R})}\frac{\E[S_{\stoppingTime}^+]}{\E[\max_{0\le k\le n}S_k]}.
    $$
    Then for $n\ge2K^2$
    $$
    C_{n,K}<\left(\frac{K}{K+1}\right)^{K+1}+\frac{6+6H_K}{H_n},
    $$
    where $H_n$ is the $n$-th harmonic number.
\end{theorem}
The upper bound stated in Theorem \ref{thm:finite-horizon-impossibility} is chosen for clarity, rather than sharpness. In the proof a stronger, but considerably less transparent bound is given. Moreover, the constants in this upper bound can be optimised with additional effort, but we do not pursue this here to preserve the readability of the argument.\footnote{Earlier investigations specific to the single sample setting, for example, obtained the smaller error term $(1+\log 2)/(2H_n)$.} 

The proof builds on the duality framework introduced in Theorem \ref{thm:finite-horizon-no-information}. We again consider a two-point increment distribution and express the resulting upper bound as a semi-infinite LP. The main novelty is that the dual objective becomes substantially more complex. We control this objective by decomposing it into several regimes and combining analytic estimates with a probabilistic reinterpretation of the relevant sums. A detailed proof sketch is provided in Appendix \ref{appendix:infinite-horizon-impossibility-proof}.



\newpage
\bibliographystyle{plainnat}
\bibliography{referencesPaperRW}

\newpage
\appendix

\section{Identities and semi-infinite programs}
Here we collect our definitions of probability distributions, and give theory on semi-infinite linear programs (LPs).
\subsection{Probability distributions and standard identities}
We collect here the distribution parametrisations and identities used throughout.
For $s>-1$ and $t>0$, the Gamma integral is
\begin{equation}
    \label{eq:gamma-integral}
    \int_0^\infty x^se^{-tx}\dd x=\frac{\Gamma(s+1)}{t^{s+1}},\quad \Gamma(s+1)=s!.
\end{equation}
For $x,y>0$ the Beta function is
$$
B(x,y):=\int_0^1t^{x-1}(1-t)^{y-1}\dd t=\frac{\Gamma(x)\Gamma(y)}{\Gamma(x+y)}.
$$
The following parametrisations are used for the Geometric, Negative Binomial and Poisson distribution:
\begin{align*}
    &X\sim \Geom(1-p) &&\Rightarrow \Prob(X=k)=(1-p)p^k,\ k\in\mathbb{N}_0,\\
    &X\sim \text{NegBin}(r,1-p) &&\Rightarrow \Prob(X=k)=\binom{k+r-1}{r-1}(1-p)^r p^k,\ k\in\mathbb{N}_0,\\
    &X\sim \text{Pois}(\lambda) &&\Rightarrow \Prob(X=k)=e^{-\lambda}\lambda^k/k!.
\end{align*}

\subsection{Weak duality of infinite linear programs} \label{sect:infty-lp-prelim}
One of our primary technical tools is 
semi-infinite LPs, in which either the number of decision variables or the number of constraints is infinite. Classical linear programming duality extends to this setting under mild regularity conditions (see, for example, \citet{popescu2005semidefinite}).

Let $A$ and $B$ be countable sets, and let $C\subset\mathbb{R}$.
Let $f : A \times B \times C \to \mathbb{R}$ and $g : C \to \mathbb{R}$ be given functions.
The general form of the primal problem is
{
\nobreakdisplays
\begin{align*}
    &\sup && t\\
    &\text{subject to} &&\sum_{a\in A}\sum_{b\in B}f(a,b,\gamma) \eta_{a}(b)+g(\gamma)\ge t\text{ for all }\gamma\in C,\\
    &&& \eta_a\in\mathcal{P}(B)\text{ for all }a \in A\text{, }t\in\mathbb{R}.
\end{align*}
}
Throughout the paper, the decision variables of each LP are listed in the final line of the formulation.
The dual variables associated with the continuum of constraints are represented by a probability measure $\mu \in \mathcal{P}(C)$. The corresponding dual problem is
{
\nobreakdisplays
\begin{align*}
    &\inf && \sum_{a\in A}\lambda(a)+\int_Cg(\gamma)\dd \mu(\gamma)\\
    &\text{subject to} &&\lambda(a)\ge\int_C f(a,b,\gamma)\dd \mu(\gamma)\text{ for all }a\in A\text{ and }b\in B,\\
    &&& \mu\in\mathcal{P}(C)\text{, }\lambda:A\to\mathbb{R}.
\end{align*}
}
By weak duality, every feasible solution of the dual provides an upper bound on the value of the primal problem. In our proofs we explicitly construct such dual feasible solutions to certify primal bounds.
\section{Independent lemmata}

\begin{lemma}
    \label{lemma:random-walk-root-bound}
    Let $(S_k)_{k\ge 0}$ be a random walk with i.i.d. increments $X_i$ having distribution
    $$
    \Prob(X_i = x) = \begin{cases}
        p,&\text{for }x=1,\\
        1-p,&\text{for } x=-L.
    \end{cases}
    $$
    Let $M=\max_{k\ge 0}S_k$ denote the maximum of the random walk. Then $\Prob(M\ge n)=\xi^n$, where $\xi=1$ if $p\ge L/(L+1)$ (positive drift), and otherwise $\xi\in\left[p,\frac{L+1}{L}p\right]$.
\end{lemma}
\begin{proof}
    When $p\ge L/(L+1)$ the drift is non-negative, so $M$ is a.s. infinite, hence $\Prob(M\ge n)=1=\xi^n$ with $\xi=1$. 
    Suppose $p< L/(L+1)$. Because $L\in\mathbb{N}$, the running maximum of the random walk can only increase in steps of $+1$, i.e., $J_i=1$ when $I_i<\infty$. Therefore, by \eqref{eq:ladder-height-decomposition}, $M=\sum_{i=1}^GJ_i=G$. Since $G$ has a geometric distribution it follows that $\Prob(M\ge n)=\xi^n$ with $\xi \in[0,1)$. 
    
    By conditioning on $X_1$ we get that $\Prob(M\ge n) = \Prob(M\ge n \ | \ X_1 = 1)\Prob(X_1 = 1) + \Prob(M\ge n \ | \ X_1 = - L)\Prob(X_1=-L)$. Because the infinite-horizon process is Markovian, it follows that $\Prob(M\ge n \ | \ X_1 = 1) = \Prob(M \geq n-1)$ and $\Prob(M\ge n \ | \ X_1 = - L) = \Prob(M \geq n+ L)$. Therefore, the tail probability satisfies the recursion
    $$
    \Prob(M\ge n)=p\Prob(M\ge n-1)+(1-p)\Prob(M\ge n+L)\implies \xi^n=p\xi^{n-1}+(1-p)\xi^{n+L}.
    $$

    Therefore, $\xi\in[0,1)$ is a root of the polynomial $f(\xi):=(1-p)\xi^{L+1}-\xi+p$.
    First the root is shown to be unique.
    Note that $f$ is continuous, $f(0)=p>0$, $f(1)=0$, $f'(\xi)=(L+1)(1-p)\xi^L-1$, and that for $p< L/(L+1)$ that $f'(1)>0$. So a root in the interval $[0,1)$ exists. Next, note that $f'(0)=-1<0$ and that $f'$ is increasing on $\xi \ge 0$, so there is exactly one root in $[0,1)$.

    To provide bounds on this root consider the sign of $f$. For a lower bound the sign is positive and for an upper bound it is negative. 
    Observe that $f(p)=(1-p)p^{L+1}>0$, so $\xi\ge p$.
    Next the upper bound. Let $x=\frac{L+1}{L}p$, then
    $
    f\left(x\right)=\frac{1}{L+1}\left((L+1-Lx)x^{L+1}-x\right)=\frac{x}{L+1}h(x)
    $, where $h(x)=x^L(L+1-Lx)-1$. Observe that $f(x)$ and $h(x)$ have the same sign. Note that $h(0)=-1$, $h(1)=0$, and that $h'(x)=L(L+1)x^{L-1}(1-x)>0$ for $x\in(0,1)$, so $h(x)\le 0$ for all $x\in[0,1]$. For $p< L/(L+1)$, $x=\frac{L+1}{L}p<1$, so $f(x)\le 0$ and $\xi\le x=\frac{L+1}{L}p$, as $f$ and $h$ have the same sign. 
\end{proof}

\begin{lemma}
    \label{lemma:discrete-simplex-volume}
    The number of non-negative $K$ tuples with a sum of $s$ is $\binom{s+K-1}{K-1}$.
\end{lemma}
\begin{proof}
    Standard fact. Prove with induction.
\end{proof}
\begin{lemma}
    \label{lemma:exponent-upper-bound}
    Let $m> 0$. If $0\le x\le 1/m$, then $(1+x)^m\le 1+2mx$. If $0\le x \le \min\{1/2,1/(2m)\}$, then $(1-x)^{-m}\le 1+4mx$.
\end{lemma}
\begin{proof}
    Note that $\log(1+x)\le x$, $-\log (1-x)\le x/(1-x)\le 2x$ for $x\in[0,1/2]$ and that $\exp(x)\le 1+2x$ for $x\in[0,1]$.
    For $0\le x\le 1/m$,
    $$
        (1+x)^m\le \exp(m\log(1+x))\le \exp(mx)\le 1+2mx.
    $$
    For $0\le x\le \min\{1/2,1/(2m)\}$,
    \begin{equation*}
        (1-x)^{-m}=\exp(-\log(1-x)m)\le \exp\left(2mx\right)\le 1+4mx.\qedhere
    \end{equation*}
\end{proof}

\begin{lemma}   
    \label{lemma:constraint-bound}
    Let $I\in\mathbb{N}$, $s\in\Nzero$, $\gamma(s,I)=\sup_{m\in\mathbb{N}}mB(s+m,I+1)$, where $B$ is the Beta function, then $\gamma(0,I)=1/(I+1)$ and for $s\ge 1$
    $$
    \gamma(s,I)=\left\lceil\frac{s}{I}\right\rceil B\left(s+\left\lceil\frac{s}{I}\right\rceil,I+1\right)\le (I-1)!\left(\frac{I}{I+1}\right)^{I+1}s^{-I}.
    $$
\end{lemma}   
\begin{proof}
    Let $f(m)=mB(s+m,I+1)$. Observe that
    \begin{align*}
        \frac{f(m+1)}{f(m)}=\frac{m+1}{m}\frac{B(s+m+1,I+1)}{B(s+m,I+1)}=\frac{m+1}{m}\frac{s+m}{s+m+I+1}.
    \end{align*}
    The last equality follows from expanding out the Beta functions.
    This final ratio is greater than or equal to $1$ if and only if
    $$
    (m+1)(s+m)\ge m(s+m+I+1) \ \iff \ s-mI\ge 0 \ \iff \  m\le s/I.
    $$
    If $s=0$, then $m\ge 1>0=s/I$, so $f$ is decreasing. Consequently, $m=1$ attains the maximum value and $\gamma(0,I)=1/(I+1)$. 
    If $s\ge 1$, 
    then $f$ is increasing for $m\le s/I$ and decreasing for $m>s/I$, so $f$ is unimodal. Moreover, its maximiser lies at this boundary, and more specifically, it lies in the set 
    then $f$ is unimodal $\{\lfloor s/I\rfloor,\lceil s/I\rceil\}$
    If $s/I$ is an integer then $\lfloor s/I\rfloor=\lceil s/I\rceil$, so $m=\lceil s/I\rceil$ attains the maximum. If $s/I$ is not an integer, then $f(\lceil s/I\rceil)\ge f(\lceil s/I\rceil-1)=f(\lfloor s/I\rfloor)$, so $m=\lceil s/I\rceil$ attains the maximum again. Therefore, 
    $$
    \gamma(s,I)=\left\lceil\frac{s}{I}\right\rceil B\left(s+\left\lceil\frac{s}{I}\right\rceil,I+1\right).
    $$
    Now the upper bound,
    \begin{align*}
        \gamma(s,I)&=\sup_{m\in\mathbb{N}}mB(s+m,I+1)=\sup_{m\in\mathbb{N}}m\frac{(s+m-1)!I!}{(s+m+I)!}\\
        &=I!\sup_{m\in\mathbb{N}}m\prod_{i=0}^{I}\frac{1}{(s+m+i)}\le I!\sup_{m\in\mathbb{N}}m(s+m)^{-(I+1)}\\
        &\le I!\frac{s}{I}(s+s/I)^{-(I+1)}=(I-1)!\left(\frac{I}{I+1}\right)^{I+1}s^{-I},
    \end{align*}
    using for the second inequality that the mapping $m\mapsto m(s+m)^{-(I+1)}$ has a maximum over the positive reals at $m=s/I$. 
\end{proof}

\begin{lemma}
    \label{lemma:pc-k-p-infimum}
    Let $\mathrm{PC}(K,p)$ be defined as in the proof of Theorem \ref{thm:partial-stopping-K-samples-prophet}, then 
    $$
    \inf_{p\in(0,1)}\mathrm{PC}(K,p)=\left(\frac{K}{K+1}\right)^{K+1}.
    $$
\end{lemma}
\begin{proof}
    To establish the result we show that
    \begin{align}
       \mathrm{PC}(K,p) \ge  \lim_{p \uparrow 1}    \mathrm{PC}(K,p) = \left(\frac{K}{K+1}\right)^{K+1}. 
     \label{eq:poiss_analysis}
    \end{align}
    We start with the equality in \eqref{eq:poiss_analysis}. Note that $q'(p)=(1+p/K)e^{-(1-p)/K}$. By using $\lim_{p \rightarrow 1} (1-q(p)(1-p)) = 1$, and L'H\^{o}pital's rule on $\lim_{p \rightarrow 1} \frac{1-p}{1-q(p)}$, we get
    $$
        \lim_{p\to 1}\mathrm{PC}(K,p)=\lim_{p\to 1}\left(\frac{1-p}{1-q(p)}\right)^{K+1}=\left(\frac{-1}{-q'(1)}\right)^{K+1}=\left(\frac{K}{K+1}\right)^{K+1}.
    $$
    To show the first inequality in \eqref{eq:poiss_analysis}, we argue that $\mathrm{PC}(K,p)$ is a decreasing function in $p$. Note that it suffices to show that the derivative of $h(p) = \log \mathrm{PC}(K,p)$ is negative. We have that
    \begin{align*}
        h'(p)&=(K+1)\left(\frac{q'(p)}{1-q(p)}-\frac{1}{1-p}\right)+\frac{q(p)-(1-p)q'(p)}{1-(1-p)q(p)}.
    \end{align*}
    We show that $(K+1)\left(\frac{q'(p)}{1-q(p)}-\frac{1}{1-p}\right)\le -1$ and $\frac{q(p)-(1-p)q'(p)}{1-(1-p)q(p)}\le 1$, so the above is $\le 0$.
    To start, $e^x\ge 1+x$, so
    \begin{align*}
        &1+\frac{1-p}{K}\le e^{\frac{1-p}{K}}\\
        \iff \ &e^{\frac{p-1}{K}}\left(\frac{K+1}{K}\left(1-p\right)+p\right)\le 1\\
        \iff \ & \frac{K+1}{K}e^{\frac{p-1}{K}}(1-p)\le 1-pe^{\frac{p-1}{K}}\\
         \iff \ & (K+1)\left(\frac{K+p}{K}\right)e^{\frac{p-1}{K}}(1-p)\le (K+p)\left(1-pe^{\frac{p-1}{K}}\right)\\
         \iff \ & (K+1)q'(p)(1-p)\le (K+p)\left(1-q(p)\right)\\
         \iff \ & (K+1)\frac{q'(p)}{1-q(p)}\le \frac{K+p}{1-p} \ \left( =\frac{K+1}{1-p}-1 \right)\\
         \iff \ & (K+1)\left(\frac{q'(p)}{1-q(p)}-\frac{1}{1-p}\right)\le -1.
    \end{align*}
    This proves the first inequality. 
    We continue with the second inequality.
    Note that for $p\in(0,1)$ that $p-\left(1-p\right)\left(1+\frac{p}{K}-p\right)\le 1$. This follows from the fact that $f(p) =p-\left(1-p\right)\left(1+\frac{p}{K}-p\right)$ is increasing for any $K \geq 1$ on $[0,1]$ and $f(1) = 1$. Observe that $e^{\frac{p-1}{K}}\in[0,1]$, so we get
    \begin{align*}
        &e^{\frac{p-1}{K}}\left(p-\left(1-p\right)\left(1+\frac{p}{K}-p\right)\right)\le 1\\
         \iff \ &pe^{\frac{p-1}{K}}-(1-p)(1+p/K-p)e^{\frac{p-1}{K}}\le 1\\
        \iff \ &q(p)-(1-p)(q'(p)-q(p))\le 1\\
         \iff \ &q(p)-(1-p)q'(p)\le 1-(1-p)q(p)\\
         \iff \ &\frac{q(p)-(1-p)q'(p)}{1-(1-p)q(p)}\le 1.\qedhere
    \end{align*}
\end{proof}

\newpage
\section{The summation lemma}
Here we give a lemma that is pivotal in bounding the objective of the dual solutions found in Theorems \ref{thm:partial-stopping-K-samples-impossibility} and \ref{thm:finite-horizon-impossibility}. 
Most of the complexity of the proof comes from establishing an $O(\log K)$ error term. An $O(K)$ error term would require significantly less precise estimates.

The proof consists of two parts. 
In the first part we bound the objective for the case where $\sum_{i=1}^Kr_i$ is large. For this large sum regime quite crude methods can be applied to stay within an $O(\log K)$ error term.
For the small sum regime such crude estimates do not suffice, and we have to take more care to stay within an $O(\log K)$ error term. 
\begin{lemma}
    \label{lemma:interior-objective}
    Let $n,K\in\mathbb{N}$ with $n>K+1+1/K$ and let the function $\gamma:\Nzero\times\mathbb{N}\to\mathbb{R}$ be as in Lemma \ref{lemma:constraint-bound}, then
    $$
    \sum_{\boldsymbol{r}\in\{0,\dots,n-1\}^K}\gamma\left(\sum_{i=1}^Kr_i,K\right)<3+3H_K +\left(\frac{K}{K+1}\right)^{K+1}H_{n-1}.
    $$
\end{lemma}
\begin{proof}
    To start, we split the sum into two parts.
    \begin{align*}
        \sum_{\boldsymbol{r}\in\{0,\dots, n-1 \}^K}\gamma\left(\sum_{i=1}^{K}r_i,K\right)&\le\sum_{\boldsymbol{r}\in\Nzero^K}\indicator{\sum_{i=1}^Kr_i\le (n-1)K}\gamma\left(\sum_{i=1}^{K}r_i,K\right)&\\
        &= \sum_{s=0}^{(n-1)K}\binom{s+K-1}{K-1}\gamma(s,K)&\text{(Lemma \ref{lemma:discrete-simplex-volume})}\\
        &=\gamma(0,K) +\sum_{j=1}^{K}\sum_{s=(j-1)K+1}^{jK}\binom{s+K-1}{K-1}\gamma(s,K)\\
        &+\sum_{s=K^2+1}^{(n-1)K}\binom{s+K-1}{K-1}\gamma(s,K).
    \end{align*}
    Note that splitting the sum is valid as long as $(n-1)K>K^2+1$, which is true because of the assumption $n>K+1+1/K$.
    The estimate used in the first line is crude, but suffices here.
    By Lemma \ref{lemma:constraint-bound} the first term $\gamma(0,K)$ is $1/(K+1)$.
    For the third term we find the following upper bound
    \begin{align*}
        \sum_{s=K^2+1}^{(n-1)K}\binom{s+K-1}{K-1}\gamma(s,K)&\le \left(\frac{K}{K+1}\right)^{K+1}\sum_{s=K^2+1}^{(n-1)K}\binom{s+K-1}{K-1}(K-1)!s^{-K}\\
        &=\left(\frac{K}{K+1}\right)^{K+1}\sum_{s=K^2+1}^{(n-1)K}\frac{(s+K-1)!}{s!}s^{-K}\\
        &\le\left(\frac{K}{K+1}\right)^{K+1}\sum_{s=K^2+1}^{(n-1)K}\left(1+\frac{K-1}{s}\right)^{K-1}\frac{1}{s}\\
        &\le\left(\frac{K}{K+1}\right)^{K+1}\sum_{s=K^2+1}^{(n-1)K}\left(1+\frac{2(K-1)^2}{s}\right)\frac{1}{s}\\
        &=\left(\frac{K}{K+1}\right)^{K+1}\left(H_{(n-1)K}-H_{K^2}+\sum_{s=K^2+1}^{(n-1)K}\frac{2(K-1)^2}{s^2}\right)\\
        &\le(H_{(n-1)K}-H_{K^2}+2)\left(\frac{K}{K+1}\right)^{K+1}.
    \end{align*}
    Using the upper bound on $\gamma$ from Lemma \ref{lemma:constraint-bound} in the first step, the fact that
    $$
    \frac{(s+K-1)!}{s!}\cdot s^{-K}
    = \prod_{j=1}^{K-1}(s+j) \cdot s^{-K}
    = \frac{1}{s}\prod_{j=1}^{K-1}\frac{s+j}{s}
    \leq \frac{1}{s}\prod_{j=1}^{K-1}\frac{s+K}{s}
    = \frac{1}{s}\left(1+\frac{K}{s}\right)^{K-1}
    $$
    in the third step, the first part of Lemma \ref{lemma:exponent-upper-bound} in the fourth step with $x = (K-1)/s$ and $m = (K-1)$ (so that indeed $x \leq 1/m$), and
    $$
        \sum_{s=K^2+1}^{(n-1)K}\frac{2(K-1)^2}{s^2}=\int_{K^2}^{(n-1)K}\frac{2(K-1)^2}{\lceil s\rceil^2} \dd s\le \int_{K^2}^{\infty}\frac{2(K-1)^2}{s^2} \dd s=\frac{2(K-1)^2}{K^2}\le 2
    $$
    in the last step. 
    
    Now the inner sum of the second term,
    \begin{align*}
        \sum_{s=(j-1)K+1}^{jK}\binom{s+K-1}{K-1}\gamma(s,K)&=j\sum_{s=(j-1)K+1}^{jK}\binom{s+K-1}{K-1}B(s+j,K+1)\\
        &=j\sum_{s=(j-1)K+1}^{jK}\binom{s+K-1}{K-1}\int_0^1x^{s+j-1}(1-x)^K\dd x\\
        &=j\int_0^1 x^{j-1}\sum_{s=(j-1)K+1}^{jK}\binom{s+K-1}{K-1}x^s(1-x)^K\dd x,
    \end{align*}
    where in the first equality, we use the fact that 
    $\frac{(j-1)K + 1}{K} \leq \frac{s}{K} \leq \frac{jK}{s}$, which implies that 
    $$j = \left\lceil \frac{(j-1)K + 1}{K} \right\rceil \leq \left\lceil \frac{s}{K} \right\rceil \leq \left\lceil \frac{jK}{K} \right\rceil = j $$
    and so $\gamma(s,K) = jB(s+j,K+1)$ by Lemma \ref{lemma:constraint-bound}; the second equality uses an equivalent definition of the Beta function; and the third equality interchanges the integral and the sum.
    
    The map $g(s) = \binom{s+K-1}{K-1}x^s(1-x)^K$ is the probability mass function (pmf) of the negative binomial distribution with parameters $K$ and $1-x$. This pmf is unimodal and has a mode at $\lfloor \frac{(K-1)x}{1-x}\rfloor$.\footnote{Unimodality follows from $g(s+1)/g(s)=x(s+K)/(s+1)$ being a decreasing function. The mode follows from determining the switch-over point.} 
    Consequently, if $x\le 1-\frac{K-1}{jK}$ then $(j-1)K+1\ge \frac{(K-1)x}{1-x}\ge \lfloor \frac{(K-1)x}{1-x}\rfloor$, meaning that the smallest value of $s$ lies right of the maximum of $g$ and so we can for every $s$ plug in the smallest value of $s$, namely $(j-1)L + 1$, leading to
    $$
    \sum_{s=(j-1)K+1}^{jK}\binom{s+K-1}{K-1}x^s(1-x)^K\le K \binom{(j-1)K+1+K-1}{K-1}x^{(j-1)K+1}(1-x)^K.
    $$
    Similarly, if $x\ge 1-\frac{K-1}{(j+1)K}$ then $jK\le\frac{(K-1)x}{1-x}-1\le \lfloor \frac{(K-1)x}{1-x}\rfloor$, so 
    $$
    \sum_{s=(j-1)K+1}^{jK}\binom{s+K-1}{K-1}x^s(1-x)^K\le K \binom{jK+K-1}{K-1}x^{jK}(1-x)^K.
    $$
    The above two cases fail to cover $x\in(1-(K-1)/(jK), 1-(K-1)/((j+1)K)$, but since the width of this interval is $O(1/j^2)$ it is sufficient to use that the integrand is less than one, as
    $$
    x^{j-1}\sum_{s=(j-1)K+1}^{jK}\binom{s+K-1}{K-1}x^s(1-x)^K\le \sum_{s=(j-1)K+1}^{jK}\binom{s+K-1}{K-1}x^s(1-x)^K\le 1,
    $$
    because the sum of probabilities is less than one. 
    These bounds can be exploited by splitting the integral into three intervals.
    \begin{align}
        &\equalSpace j\int_0^1 x^{j-1}\sum_{s=(j-1)K+1}^{jK}\binom{s+K-1}{K-1}x^s(1-x)^K\dd x \nonumber \\
        &=j\int_0^{1-\frac{K-1}{jK}} x^{j-1}\sum_{s=(j-1)K+1}^{jK}\binom{s+K-1}{K-1}x^s(1-x)^K\dd x \nonumber \\
        &+j\int_{1-\frac{K-1}{jK}}^{1-\frac{K-1}{(j+1)K}} x^{j-1}\sum_{s=(j-1)K+1}^{jK}\binom{s+K-1}{K-1}x^s(1-x)^K\dd x \nonumber \\
        &+j\int_{1-\frac{K-1}{(j+1)K}}^1 x^{j-1}\sum_{s=(j-1)K+1}^{jK}\binom{s+K-1}{K-1}x^s(1-x)^K\dd x \nonumber \\
        &\le j\int_0^{1-\frac{K-1}{jK}} x^{j-1}K\binom{(j-1)K+1+K-1}{K-1}x^{(j-1)K+1}(1-x)^K\dd x \nonumber \\
        &+j\left(1-\frac{K-1}{(j+1)K}-\left(1-\frac{K-1}{jK}\right)\right) \nonumber \\
        &+j\int_{1-\frac{K-1}{(j+1)K}}^1 x^{j-1}K\binom{jK+K-1}{K-1}x^{jK}(1-x)^K\dd x \nonumber \\
        &\le jK\binom{jK}{K-1}\int_0^1 x^{j+(j-1)K}(1-x)^K\dd x \nonumber \\
        & +\frac{K-1}{K}\frac{1}{j+1} \nonumber \\
        &+jK\binom{jK+K-1}{K-1}\int_0^1 x^{j(K+1)-1}(1-x)^K\dd x. \label{eq:final_three}
    \end{align}
    For the middle interval we make use of the fact that the integrand is less than one. The second inequality follows from integrating over a larger domain. 
    
    Finally, we bound the three quantities in \eqref{eq:final_three}.
    For the first term we split cases $j=1$ and $j\ge 2$. For $j=1$,
    $$
    K\binom{K}{K-1}\int_0^1 x(1-x)^K\dd x=\frac{K^2}{(K+2)(K+1)}\le1.
    $$
    For $j\ge 2$ it is
    \begin{align*}
        &\equalSpace jK\binom{jK}{K-1}\int_0^1 x^{j+(j-1)K}(1-x)^K\dd x\\&=jK\binom{jK}{K-1}\frac{(j+(j-1)K)!K!}{(jK+j+1)!}\\
        &=jK^2\frac{(jK-K+j)!}{(jK-K+1)!}\frac{(jK)!}{(jK+j+1)!}\\
        &=\frac{jK^{2}}{\left(jK-K+j+2\right)\left(jK-K+j+1\right)}\frac{(jK-K+j+2)!}{(jK-K+1)!}\frac{(jK)!}{(jK+j+1)!}\\
        &=\frac{jK^{2}}{\left(jK-K+j+2\right)\left(jK-K+j+1\right)}\prod_{i=1}^{j+1}\frac{jK-K+1+i}{jK+i}\\
        &\le\frac{jK^{2}}{\left(jK-K+j+2\right)\left(jK-K+j+1\right)}\\
        &\le\frac{jK^{2}}{\left(jK-K\right)\left(jK-K\right)}=\frac{j}{(j-1)^2}
    \end{align*}
    Lastly, the third term works the same as above
    \begin{align*}
        &\equalSpace jK\binom{jK+K-1}{K-1}\int_0^1 x^{j(K+1)-1}(1-x)^K\dd x\\&=jK\binom{jK+K-1}{K-1}\frac{(j(K+1)-1)!K!}{(j(K+1)+K)!}\\
        &=jK^{2}\frac{\left(jK+K-1\right)!}{(jK+K+j)!}\frac{(jK+j-1)!}{\left(jK\right)!}\\
        &=\frac{jK^{2}}{\left(jK+j+1\right)\left(jK+j\right)}\prod_{i=1}^{j+1}\frac{jK+i}{jK+K-1+i}\\
        &\le\frac{jK^{2}}{\left(jK+j+1\right)\left(jK+j\right)}\le\frac{jK^{2}}{\left(jK\right)\left(jK\right)}=\frac{1}{j}.
    \end{align*}
    So we obtain the bound $2+(K-1)/(2K)$ for $j=1$ and $\frac{j}{(j-1)^2}+\frac{K-1}{K(j+1)}+\frac{1}{j}$ for $j\ge 2$.
    Adding these bounds yields
    \begin{align*}
        \sum_{j=1}^{K}\sum_{s=(j-1)K+1}^{jK}\binom{s+K-1}{K-1}\gamma(s,K)&\le2+\frac{K-1}{2K}+\sum_{j=2}^{K}\left(\frac{j}{(j-1)^2}+\frac{K-1}{K(j+1)}+\frac{1}{j}\right)\\
        &\le\frac{5}{2}+\sum_{j=2}^{K}\left(\frac{1}{j-1}+\frac{1}{(j-1)^2}+\frac{1}{j+1}+\frac{1}{j}\right)\\
        &=H_{K-1}+H_K+H_{K+1}+\sum_{j=1}^{K-1}\frac{1}{j^2}\\
        &\le3H_{K}+\pi^2/6,
    \end{align*}
    using $H_{K-1}+H_{K+1}\le 2H_K$ in the last step.
    Lastly, adding all three terms together yields
    \begin{align*}
        \sum_{\boldsymbol{r}\in\{0,\dots,n-1\}^K}\gamma\left(\sum_{i=1}^Kr_i,K\right)<\frac{1}{K+1}+(H_{(n-1)K}-H_{K^2}+2)\left(\frac{K}{K+1}\right)^{K+1} +3H_K+\pi^2/6
    \end{align*}
    Under a minor loss, we obtain our final clean bound by using $H_{(n-1)K}-H_{K^2}\le H_{n-1}+H_{K}-H_{K^2}\le H_{n-1}$:
    \begin{align*}
        &\equalSpace \frac{1}{K+1}+(H_{(n-1)K}-H_{K^2}+2)\left(\frac{K}{K+1}\right)^{K+1} +3H_K+\pi^2/6\\
        &\le \frac{1}{K+1}+2\left(\frac{K}{K+1}\right)^{K+1}+\pi^2/6 +3H_K+\left(\frac{K}{K+1}\right)^{K+1}H_{n-1} \\
        &< 3+3H_K +\left(\frac{K}{K+1}\right)^{K+1}H_{n-1}< 3+3H_K +\left(\frac{K}{K+1}\right)^{K+1}H_{n} 
    \end{align*}
    Using for the penultimate bound that $1/(K+1) \leq 1/2$ and $2(K/(K+1))^{K+1} \leq 2/e$.
\end{proof}
\newpage

\section{Proof of Theorem \ref{thm:partial-stopping-K-samples-impossibility}}
\label{appendix:infinite-horizon-impossibility-proof}
First, a proof sketch. Informally, we want to use increments $+1$ and $-\infty$ such that the process stays below zero indefinitely after the $-\infty$ increment. Such a process is much more amenable to analysis, because it has no variability. Specifically, if the $-\infty$ increment occurs with probability $1-p$, then $S_k^+=k\indicator{N\ge k}$, where $N\sim\Geom(1-p)$.
    
Using $-\infty$ increments is not permitted, however, so we approximate it by means of  a $-L$ increments with $L$ large.
This approximation works, but it incurs a loss $\epsilon(p, L)$.    
The details of this reduction are in Lemma \ref{lemma:infinite-horizon-samples-finite-reduction}.
    
The next step is writing the resulting upper bound on the prophet constant as a semi-infinite LP. The resulting program is very similar to the program we derive for Theorem \ref{thm:finite-horizon-impossibility}, and most the analysis of it is borrowed from the proof of that theorem. All that is needed for proving Theorem \ref{thm:partial-stopping-K-samples-impossibility} is control on the $\epsilon(p,L)$ term. We do this with the tail bound from Lemma \ref{lemma:random-walk-root-bound} and the analysis in Lemma \ref{lemma:epsilon-integral-bound}.

We first state our two main lemmas, with their proofs deferred to the end of this section.
\begin{lemma}
\label{lemma:infinite-horizon-samples-finite-reduction}
    Let $L\in\mathbb{N}$, $p\in[0,L/(L+1))$ and let $(S_k)_{k\ge 0}$ be a random walk with increments $X_t$ following
    $$
    \Prob(X_t = x) = \begin{cases}
        p,&\text{for }x=1,\\
        1-p,&\text{for } x=-L.
    \end{cases}
    $$
    For $i=1,\dots,K$ let $(S_{k,i})_{k\ge 0}$ be independent copies, and let $\stoppingRule$ be a stopping rule that is based solely on the (realised) rewards $(S_1^+,\dots,S_k^+)$ and sample data $D=(S_{t,i}^+)_{t\ge 0,1\le i\le K}$. Then
    \begin{align*}
        \frac{\E[S_{\stoppingTime}^+]}{\E[\max_{k\ge 0}S_k]}&\le\frac{1}{\E[M]}\sum\limits_{\substack{L > j_i \ge 0 \\ i = 1,\ldots,K}}\left(\prod_{i=1}^K\Prob(N_i=j_i)\right)
        \sum_{l=1}^{L-1}l\Prob(N\ge l)\Prob(Q_{j_1,\dots,j_K}=l)+\epsilon(p,L).
    \end{align*}
    In the above equation
    \begin{itemize}
        \item $M=\max_{k\ge0}S_k$ is the prophets payoff;
        \item $N,N_1,\dots,N_K\iid\text{Geom}(1-p)$ indicate the initial run of $+1$'s;
        \item $
    \Prob(Q_{j_1,\dots,j_K}=l)=\Prob(\stoppingTime=l|D=((0,1,\dots,j_1,0,\dots),\dots,(0,1,\dots,j_K,0,\dots)), S_l=l)
    $ is the probability of stopping at index $l$ given the sample paths and the fact that we make it to $l$. 
        \item Lastly, the error term is given by
    $$
        \epsilon(p, L)=\Prob(N\ge L)+(K+1)\sum_{n=0}^\infty\Prob(N=n)\Prob(M\ge L-n).
    $$
    \end{itemize}
\end{lemma}

\begin{lemma}
    \label{lemma:epsilon-integral-bound}
    Let $\epsilon(p,L)$ be as defined in Lemma \ref{lemma:infinite-horizon-samples-finite-reduction} and let 
    $\dd \nu_n(p)=\frac{1}{1-p}\dd p$, then
    $$
    \int_{0}^{\frac{L}{L+1}}\epsilon(p,L)\mathrm{d}\nu_n(p)<K+1+1/e.
    $$
\end{lemma}
With the above two lemmas the theorem can be proven.
\begin{cleverproof}{thm:partial-stopping-K-samples-impossibility}
    The prophet constant can be written as
    $$
    \mathrm{PC}(K)=\sup_{\stoppingRule\in\mathcal{R}(\mathcal{D})}\inf_{F\in\mathcal{P}(\mathbb{R})}\frac{\E[S^+_{\stoppingTime}]}{\E[\max_{k\ge 0}S_k]}.
    $$
    Fix $L\in\mathbb{N}$ and let $B=\{\mu\in\mathcal{P}(\mathbb{R}):\mu(+1)=p,\mu(-L)=1-p,p\in(0,L/(L+1))\}$. By Lemma \ref{lemma:infinite-horizon-samples-finite-reduction}
    \begin{align*}
        \mathrm{PC}(K)&\le \sup_{\stoppingRule\in\mathcal{R}(\mathcal{D})}\inf_{F\in B}\frac{\E[S^+_{\stoppingTime}]}{\E[\max_{k\ge 0}S_k]}\\
        &\le\sup_{\stoppingRule\in\mathcal{R}(\mathcal{D})}\inf_{F\in B}\left\{\frac{1}{\E[M]}\sum\limits_{\substack{L > j_i \ge 0 \\ i = 1,\ldots,K}}\left(\prod_{i=1}^K\Prob(N_i=j_i)\right)
        \sum_{l=1}^{L-1}l\Prob(N\ge l)\Prob(Q_{j_1,\dots,j_K}=l)+\epsilon(p,L)\right\}
    \end{align*}
    inheriting all the notation used in (the proof of) Lemma \ref{lemma:infinite-horizon-samples-finite-reduction}. 
    The probabilities $\Prob(Q_{j_1,\dots,j_K}=l)$ are defined implicitly through the stopping rule, so instead of taking a supremum over the stopping rules, the supremum can be taken over these probabilities directly, which we denote by $\boldsymbol{q}$. Moreover, using the observation that $\E[M]\ge \E[N]$, this yields
    \begin{align*}
        \mathrm{PC}(K)
        &\le\sup_{\boldsymbol{q}}\inf_{F\in B}\left\{\frac{1}{\E[N]}\sum\limits_{\substack{L > j_i \ge 0 \\ i = 1,\ldots,K}}\left(\prod_{i=1}^K\Prob(N_i=j_i)\right)
        \sum_{l=1}^{L-1}l\Prob(N\ge l)\Prob(Q_{j_1,\dots,j_K}=l)+\epsilon(p,L)\right\}\\
        &=\sup_{\boldsymbol{q}}\inf_{F\in B}\left\{\frac{1-p}{p}\sum\limits_{\substack{L > j_i \ge 0 \\ i = 1,\ldots,K}}(1-p)^Kp^{\sum_{i=1}^Kj_i}
        \sum_{l=1}^{L-1}lp^l\Prob(Q_{j_1,\dots,j_K}=l)+\epsilon(p,L)\right\}.
    \end{align*}
    Note that in optimality that there is no slack, so we have $\sum_{l=1}^{L-1}\Prob(Q_{j_1,\dots,j_K}=l)=1$ in optimality. Consequently, the above upper bound on the prophet constant can be reformulated as a semi-infinite LP as follows 
    {
    \nobreakdisplays
    \begin{align*}
        &\sup && C\\
        &\text{subject to} &&\sum\limits_{\substack{L > j_i \ge 0 \\ i = 1,\ldots,K}}\sum_{l=1}^{L-1}l(1-p)^{K+1}p^{l-1+\sum_{i=1}^Kj_i}\Prob(Q_{j_1,\dots,j_K}=l)
         +\epsilon(p,L)\ge C\;\forall p\in\left(0,\frac{L}{L+1}\right),
        \\
        &&& \kappa(\boldsymbol{j})\in\mathcal{P}(\{1,\dots,L-1\})\text{ for all }\boldsymbol{j}\in\{0,\dots,L-1\}^K\text{, }C\in\mathbb{R},
    \end{align*}
    }
    where $Q_{\boldsymbol{j}}\sim\kappa(\boldsymbol{j})$. The dual of the above LP is\footnote{It follows from the general form provided in Section \ref{sect:infty-lp-prelim}}
    {
    \nobreakdisplays
    \begin{align*}
        &\inf && \sum_{\boldsymbol{j}\in\{0,\dots,L-1\}^K}\lambda(\boldsymbol{j}) + \int_0^{\frac{L}{L+1}}\epsilon(p,L)\dd\mu(p)\\
        &\text{subject to} &&\lambda(\boldsymbol{j})\ge\int_0^\frac{L}{L+1}l(1-p)^{K+1}p^{l-1+\sum_{i=1}^Kj_i}\dd\mu(p)\\
        &&&\text{for all }\boldsymbol{j}\in \{0,\dots,L-1\}^K\text{ and }l\in\{1,\dots,L-1\},
        \\
        &&& \mu\in\mathcal{P}\left(0,\frac{L}{L+1}\right),\lambda:\{0,\dots,L-1\}^K\to\mathbb{R},
    \end{align*}
    }
    where the measure $\mu$ is associated with the continuum of constraints. By weak duality, any feasible dual solution gives an upper bound on the dual and therefore $\mathrm{PC}(K)$. 
    
    We will next construct a feasible dual solution $(\nu,(\zeta(\boldsymbol{j})_{\boldsymbol{j}})$. Consider the measure 
    $$
    \dd \nu(p)=\frac{1}{\log(1+L)}\frac{1}{1-p}\dd p \ \text{ and } \ \zeta(\boldsymbol{j})=\beta\left(\sum_{i=1}^Kj_i\right),
    $$
    where 
    \begin{align}
        \label{eq:a-beta-bound}
        \beta(s):=&\sup_{l\in\{1,\dots,L-1\}}\int_0^\frac{L}{L+1}l(1-p)^{K+1}p^{l-1+s}\dd\nu(p)\nonumber\\
        &=\sup_{l\in\{1,\dots,L-1\}}\frac{l}{\log(1+L)}\int_0^\frac{L}{L+1}(1-p)^{K}p^{l-1+s}\dd p\nonumber\\
        &\le \sup_{l\in\{1,\dots,L-1\}}\frac{l}{\log(1+L)} B(s+l,K+1)\nonumber\\
        &\le \frac{1}{\log(1+L)}\sup_{l\in\mathbb{N}}l B(s+l,K+1)
    \end{align}   
    Note that $\nu\in\mathcal{P}((0,L/(L+1))$ as $\int_0^{L/(L+1)}\dd \nu(p)=1$.
    By constructing $\zeta$ as we did, the constraint
    $$
    \zeta(\boldsymbol{j})\ge \int_0^\frac{L}{L+1}l(1-p)^{K+1}p^{l-1+s}\dd\nu(p)
    $$
    is satisfied for all $l$ and $\boldsymbol{j}$. As a result, the pair $(\nu,(\zeta(\boldsymbol{j})_{\boldsymbol{j}})$ is feasible. 
    
    What remains is calculating (an upper bound on) the associated objective. Using weak duality in the first step, we find
    \begin{align*}
        \mathrm{PC}(K)&\le\sum_{\boldsymbol{j}\in\{0,\dots,L-1\}^K}\zeta(\boldsymbol{j}) + \int_0^{\frac{L}{L+1}}\epsilon(p,L)\dd\nu(p) \\
        &\le \sum_{\boldsymbol{j}\in\{0,\dots,L-1\}^K}\zeta(\boldsymbol{j}) + \frac{K+1+1/e}{\log(1+L)}\\
        &\le \frac{3+3H_K+\left(\frac{K}{K+1}\right)^{K+1}H_{L-1}}{\log (1+L)} + \frac{K+1+1/e}{\log(1+L)}\\
        &\le \frac{3+3H_K+\left(\frac{K}{K+1}\right)^{K+1}(1+\log(1+L))}{\log (1+L)} + \frac{K+1+1/e}{\log(1+L)}\\
        &\le \left(\frac{K}{K+1}\right)^{K+1}+ \frac{K+3H_K+3+2/e}{\log(1+L)},
    \end{align*}
    using Lemma \ref{lemma:epsilon-integral-bound} for the second step.
    For the third step we use Lemma \ref{lemma:interior-objective} by noting that by \eqref{eq:a-beta-bound} that $\beta(s)\le \frac{\gamma(s,K)}{\log(1+L)}$, where $\gamma$ is an in the statement of Lemma \ref{lemma:interior-objective}.
    The fourth and fifth step use that $H_{L-1} \leq 1 + \log(1+L)$ and $(K/(K+1))^{K+1} \leq 1/e$, respectively.
    Taking $L\to\infty$ shows that $\mathrm{PC}(K)\le (K/(K+1))^{K+1}$.
\end{cleverproof}

\subsection{Remaining proofs}
We finish with the deferred proofs.
\begin{cleverproof}{lemma:infinite-horizon-samples-finite-reduction}
    Start by defining
    \begin{itemize}
        \item $M=\max_{k\ge 0}S_k$ : The overall maximum of the (true) walk;
        \item $N=\max\{k:S_j=j\text{ for all }j\le k\}$ : The moment at which the (true) walk has its first increment of $-L$;
        \item $N_i=\max\{k:S_{j,i}=j\text{ for all }j\le k\}$, The moment at which the $i$-th sampled walk has its first increment of $-L$.
        \item $A_i=\{S_{k,i}<0\text{ for all }k>N_i\}$ : The event that the $i$-th sampled walk does not exceed $0$ after its first increment of $-L$.
        \item $A=\bigcap_{i=1}^KA_i$. The event that none of the sampled walks exceed $0$ after the first increment of $-L$.
    \end{itemize}
    Take note of the following decomposition
    \begin{align}
        \E[S_{\stoppingTime}^+]=\E[S_{\stoppingTime}^+\mathds{1}_{\{A^C\}}]+\E[S_{\stoppingTime}^+\mathds{1}_{\{A\}}].
        \label{eq:random_walk_proof_split}
    \end{align}
    Using the union bound the first term is bounded as follows
    \begin{align}
        \E[S_{\stoppingTime}^+\mathds{1}_{\{A^C\}}]&\le \E[M\mathds{1}_{\{A^C\}}]= \Prob(A^C)\E[M]=\Prob\left(\bigcup_{i=1}^KA_i^C\right)\E[M]\nonumber\\
        &\le\sum_{i=1}^K\Prob(A_i^C)\E[M] =K\Prob(A_1^C)\E[M].\label{eq:random_walk_unionbound}
    \end{align}
    In the first inequality, the reward of the stopping time is (coarsely) upper bounded by the maximum $M$ of the (true) random walk. The second inequality is the union bound.

    Later, by sending $L\to\infty$, the term $\Prob(A_1^C)$ will be send to to zero. Consequently, only the second term in \eqref{eq:random_walk_proof_split} remains. This term is easier to handle, as conditional on the event $A$, the problem becomes much more tractable.
    For brevity, let $\tau(j_1,\dots,j_K)$ denote $\stoppingTime|D=((0,1,2,\dots,j_1,0,\dots),\dots,(0,1,2,\dots,j_K,0,\dots))$. The second term in \eqref{eq:random_walk_proof_split} can be bounded as follows:
    \begin{align*}
        \E[S_{\stoppingTime}^+\mathds{1}_{\{A\}}]&=\sum_{j_1\ge0,\dots,j_K\ge 0}\E\left[S_{\stoppingTime}^+\indicator{A\cap\bigcap_{i=1}^K\{N_i=j_i\}}\right]\\
        &=\sum_{L>j_1\ge0,\dots,L>j_K\ge 0}\E\left[S_{\stoppingTime}^+\indicator{A\cap\bigcap_{i=1}^K\{N_i=j_i\}}\right]\\
        &=\sum_{L>j_1\ge0,\dots,L>j_K\ge 0}\Prob\left(A\cap\bigcap_{i=1}^K\{N_i=j_i\}\right)\E\left[S_{\tau(j_1,\dots,j_K)}^+\right]\\
        &\le\sum_{L>j_1\ge0,\dots,L>j_K\ge 0}\left(\prod_{i=1}^K\Prob(N_i=j_i)\right)\E\left[S_{\tau(j_1,\dots,j_K)}^+\right]
    \end{align*}
    The truncation in the second step follows from the fact that the event $A_i$ cannot occur if $N_i\ge L$, as $S_{N_i+1,i}\ge 0$ in that case. 

    To bound the term $\E\left[S_{\tau(j_1,\dots,j_K)}^+\right]$, we first bound $\E[S_{\tau(j_1,\dots,j_K)}^+|N=n]$. Rewriting yields that
    \begin{align*}
        \E[S_{\tau(j_1,\dots,j_K)}^+|N=n]&=\E\left[S_{\tau(j_1,\dots,j_K)}^+\mathds{1}_{\{\tau(j_1,\dots,j_K)>n\}}+\sum_{l=0}^n S_l^+\mathds{1}_{\{\tau(j_1,\dots,j_K)=l\}}|N=n\right]\\
        &=\sum_{l=0}^n\E\left[ S_l^+\mathds{1}_{\{\tau(j_1,\dots,j_K)=l\}}|X_1=1,\dots,X_l=1\right]\\
        &+\E\left[S_{\tau(j_1,\dots,j_K)}^+\mathds{1}_{\{\tau(j_1,\dots,j_K)>n\}}|X_1=1,\dots,X_n=1,X_{n+1}=-L\right]\\
        &=\sum_{l=1}^n l \Prob(\tau(j_1,\dots,j_K)=l|X_1=1,\dots,X_l=1)\\
        &+\E\left[S_{\tau(j_1,\dots,j_K)}^+\mathds{1}_{\{\tau(j_1,\dots,j_K)>n\}}|X_1=1,\dots,X_n=1,X_{n+1}=-L\right].
    \end{align*}
    The conditioning in the first term can be written succinctly as $S_l=l$, and for the second term
    \begin{align*}
        &\equalSpace\E\left[S_{\tau(j_1,\dots,j_K)}^+\mathds{1}_{\{\tau(j_1,\dots,j_K)>n\}}|X_1=1,\dots,X_n=1,X_{n+1}=-L\right]\\
        &\le \E\left[\max_{k\ge n+1}S^+_k\mathds{1}_{\{\tau(j_1,\dots,j_K)>n\}}|X_1=1,\dots,X_n=1,X_{n+1}=-L\right]\\
        &\le \E\left[\max_{k\ge n+1}S^+_k|X_1=1,\dots,X_n=1,X_{n+1}=-L\right]\\
        &= \E\left[\max_{k\ge n+1}S^+_k|S_{n+1}=n-L\right]=\E[(M'+(n-L))^+],
    \end{align*}
    where $M'$ is an i.i.d. copy of $M$. Note that both $M$ and $M'$ have a geometric distribution (see Lemma \ref{lemma:random-walk-root-bound}) and that $\E[(M'+(n-L))^+] = \E[(M+(n-L))^+]$, but we use the i.i.d. copy $M'$ to emphasise that it should be viewed as a (distributionally equivalent) random walk that starts at $n - L$.
    Taken together, 
    \begin{align}
        \E[S_{\tau(j_1,\dots,j_K)}^+|N=n]\le\sum_{l=1}^nl\Prob(\tau(j_1,\dots,j_K)=l|S_l=l)+\E[(M'+(n-L))^+].
        \label{eq:random_walk_cond_bound}
    \end{align}
    Next, the above conditional bound is turned into an unconditional bound.
    \begin{align*}
        \E[S_{\tau(j_1,\dots,j_K)}^+]&=\sum_{n=0}^\infty\Prob(N=n)\E[S_{\tau(j_1,\dots,j_K)}^+|N=n]  \\
        &\le \sum_{n=0}^{L-1}\Prob(N=n)\E[S_{\tau(j_1,\dots,j_K)}^+|N=n]+\Prob(N\ge L)\E[M]  \\
        &\le \sum_{n=0}^{L-1}\Prob(N=n)\Bigg[\sum_{l=1}^nl\Prob(\tau(j_1,\dots,j_K)=l|S_l=l)\\&+\E[(M'+(n-L))^+]\Bigg]  +  \Prob(N\ge L)\E[M]  \\
        &= \sum_{n=0}^{L-1}\Prob(N=n)\sum_{l=1}^nl\Prob(\tau(j_1,\dots,j_K)=l|S_l=l)  \\
        &+\sum_{n=0}^{L-1}\Prob(N=n)\E[(M'+(n-L))^+]   +  \Prob(N\ge L)\E[M].
    \end{align*}
    In the second step, the reward of the stopping time is upper bounded by the maximum of the random walk. In the third step, \eqref{eq:random_walk_cond_bound} is plugged in.
    By swapping the order of summation we find
    \begin{align*}
        &\equalSpace\sum_{n=0}^{L-1}\Prob(N=n)\sum_{l=1}^nl\Prob(\tau(j_1,\dots,j_K)=l|S_l=l)\\
        & = \sum_{l=1}^{L-1} l\Prob(\tau(j_1,\dots,j_K)=l|S_l=l) \sum_{n=l}^{L-1} \Prob(N=n)\\
        &= \sum_{l=1}^{L-1}l\Prob(\tau(j_1,\dots,j_K)=l|S_l=l)(\Prob(N\ge l)-\Prob(N\ge L)) \\
        &\le \sum_{l=1}^{L-1}l\Prob(N\ge l)\cdot\Prob(\tau(j_1,\dots,j_K)=l|S_l=l).
    \end{align*}
    Taken together this yields
    \begin{align}
        \E[S_{\tau(j_1,\dots,j_K)}^+]&\le  \sum_{l=1}^{L-1}l\Prob(N\ge l)\cdot\Prob(\tau(j_1,\dots,j_K)=l|S_l=l) \nonumber \\
        &  +\sum_{n=0}^{L-1}\Prob(N=n)\E[(M'+(n-L))^+]+\Prob(N\ge L)\E[M].
         \label{eq:random_walk_uncond_bound}
    \end{align}
    To bound the second and third term in the final right hand side of \eqref{eq:random_walk_uncond_bound}, recall that $\Prob(A_1^C)$ (which is the same as $\Prob(A_i^C)$ for every $i$) is the probability that the first sampled random walk becomes non-negative after the first increment of $-L$, therefore 
    \begin{align}
        \Prob(A_1^C)=\sum_{n=0}^\infty\Prob(N_1=n)\Prob(M'\ge L-n).
        \label{eq:random_walk_ai_complement}
    \end{align}
    If $G$ is a geometrically distributed random variable and $k\in\mathbb{N}_0$, then by the memorylessness property
    \begin{equation}
    \label{eq:geometric-stop-loss}
        \E[(G-k)^+]=\E[(G-k)\indicator{G\ge k}]=\Prob(G\ge k)\E[G-k|G\ge k]=\Prob(G\ge k)\E[G].
    \end{equation}
    With this we find that
    \begin{align}
        &\equalSpace\sum_{n=0}^{L-1}\Prob(N=n)\E[(M'+(n-L))^+]+\Prob(N\ge L)\E[M] \nonumber \\
        &=\sum_{n=0}^{L-1}\Prob(N=n)\Prob(M'\ge L-n)\E[M']+\Prob(N\ge L)\E[M] \nonumber \\
        &=\left(\Prob(N\ge L)+\sum_{n=0}^{L-1}\Prob(N=n)\Prob(M'\ge L-n)\right)\E[M] \nonumber \\
        &\le (\Prob(N\ge L)+\Prob(A_1^C))\E[M],
        \label{eq:random_walk_ai}
    \end{align}
    using \eqref{eq:geometric-stop-loss} in the second step. Combining the bounds in \eqref{eq:random_walk_unionbound}, \eqref{eq:random_walk_uncond_bound} and \eqref{eq:random_walk_ai} to bound the right hand side quantities in \eqref{eq:random_walk_proof_split} yields
    \begin{align*}
        \E[S_{\stoppingTime}^+]&\le \sum_{\substack{L > j_i \ge 0 \\ i = 1,\ldots,K}}\left(\prod_{i=1}^K\Prob(N_i=j_i)\right)
        \Bigg(\sum_{l=1}^{L-1}l\Prob(N\ge l)\cdot\Prob(\tau(j_1,\dots,j_K)=l|S_l=l)\\&+(\Prob(N\ge L)+\Prob(A_i^C))\E[M]\Bigg)+K\Prob(A_i^C)\E[M]\\
        &\le \sum_{\substack{L > j_i \ge 0 \\ i = 1,\ldots,K}}\left(\prod_{i=1}^K\Prob(N_i=j_i)\right)
        \sum_{l=1}^{L-1}l\Prob(N\ge l)\cdot\Prob(\tau(j_1,\dots,j_K)=l|S_l=l)\\
        &+(\Prob(N\ge L)+(K+1)\Prob(A_1^C))\E[M].
    \end{align*}
    Dividing by $\E[\max_{k\ge 0}S_k]$ and plugging in the expression in \eqref{eq:random_walk_ai_complement} for $\Prob(A_1^C)$ gives the result.
\end{cleverproof}

\begin{cleverproof}{lemma:epsilon-integral-bound}
    Recall that 
    $$
        \epsilon(p,L)=\Prob(N\ge L)+(K+1)\sum_{n=0}^\infty\Prob(N=n)\Prob(M\ge L-n).
    $$
    By Lemma \ref{lemma:random-walk-root-bound}, $\Prob(M\ge k)\le (\frac{L+1}{L}p)^k$. Using that $N\sim\Geom(1-p)$ we find that
    \begin{align*}
        \epsilon(p,L)&\le\Prob(N\ge L)+(K+1) \left(\Prob(N\ge L)+\sum_{n=0}^{L-1}\Prob(N=n)\left(\frac{L+1}{L}p\right)^{L-n}\right)\\
        &=p^L+(K+1)\left(p^L+\sum_{n=0}^{L-1}(1-p)p^n\left(\frac{L+1}{L}p\right)^{L-n}\right)\\
        &=p^L+(K+1)\left(p^L+(1-p)p^L\sum_{n=0}^{L-1}\left(\frac{L+1}{L}\right)^{L-n}\right)\\
        &=p^L+(K+1)\left(p^L+(1-p)p^L\left(\left(\frac{L+1}{L}\right)^{L}-1\right)\left(L+1\right)\right)\\
        &\le p^L+(K+1)\left(p^L+(e-1)(1-p)p^L\left(L+1\right)\right)\\
        &=(K+2)p^L+(e-1)(K+1)(L+1)(1-p)p^L.
    \end{align*}
    Using $(1+1/x)^x\le e$ for $x\ge 0$ in the penultimate step.
    With the above bound the integral inequality can be established.
    \begin{align*}
        \int_{0}^{\frac{L}{L+1}}\epsilon(p,L)\mathrm{d}\nu_n(p)&\le \int_{0}^{\frac{L}{L+1}}\left((K+2)p^L+(e-1)(K+1)(L+1)(1-p)p^L\right)\mathrm{d}\nu_n(p)\\
       & = \int_{0}^{\frac{L}{L+1}}\left((K+2)p^L+(e-1)(K+1)(L+1)(1-p)p^L\right)\frac{1}{1-p}\dd p\\
       & = \int_{0}^{\frac{L}{L+1}}\left(\frac{K+2}{1-p}+(e-1)(K+1)(L+1)\right)p^L\dd p\\
       & \le \int_{0}^{\frac{L}{L+1}}\left(\frac{K+2}{1-L/(L+1)}+(e-1)(K+1)(L+1)\right)p^L\dd p\\
       & = (eK+e+1) \int_{0}^{\frac{L}{L+1}}(L+1)p^L\dd p\\
       &=(eK+e+1)\left(\frac{L}{L+1}\right)^{L+1}\le K+1+1/e.
    \end{align*}
    Using  $(1-1/x)^x\le 1/e$ for $x\ge 1$ in the last step.
\end{cleverproof}
\newpage
\section{Proof of Theorem \ref{thm:finite-horizon-impossibility}}
First, a proof sketch. 
The proof is more like Theorem \ref{thm:finite-horizon-no-information} than Theorem \ref{thm:partial-stopping-K-samples-impossibility}, as no random walk error term like $\epsilon(p,L)$ is required.
To start, consider an increment distribution with $+1$ and $-n$ increments and write the resulting upper bound as a semi-infinite LP.
Using the same dual measure as in Theorem \ref{thm:finite-horizon-no-information} we obtain a bound.
In Theorem \ref{thm:finite-horizon-no-information} this all but completes the proof as computing the objective \eqref{eq:no-info-dual} only consists of a single term.
This is the hard step in Theorem \ref{thm:finite-horizon-impossibility}, as the dual objective contains $(n+1)^K$ terms in the form of $$
\sum_{\boldsymbol{r}\in\{0,\dots, n \}^K}\beta\left(\sum_{i=1}^Kr_i,\sum_{i=1}^K\
\indicator{r_i < n}\right).
$$
The second argument represents the amount of samples that are within the horizon limit.
This sum is then split into two parts: All $r_i<n$, and at least one $r_i=n$.
The analysis of the first part offloaded to Lemma \ref{lemma:interior-objective}.
To compute the second part we cleverly interpret the sum over the hypercube as an expectation of (independent) discrete uniform random variables and by then approximating them by their uniform counterpart 
This turns the daunting sum into a tractable integral that is further simplified by using the moment generating function of the sum of uniform random variables.

\begin{proof}[Proof of Theorem \ref{thm:finite-horizon-impossibility}]

Consider, just as in the proof of Theorem \ref{thm:finite-horizon-no-information}, an increment distribution that assigns probability $p$ to $+1$ and probability $1-p$ to $-n$. 
Let $M_{n,i}=\max_{0\le k \le n}S_{k,i}$.
By generalising \eqref{eq:tau-write-up} we obtain that
\begin{align*}
    \E[S_\tau^+]&=\sum_{\boldsymbol{r}\in\{0,\dots,n\}^K}\left(\prod_{i=1}^K\Prob(M_{n,i}=r_i)\right)\E[S_\tau^+|M_{n,1}=r_1,\dots,M_{n,K}=r_K]\\
    &=\sum_{\boldsymbol{r}\in\{0,\dots,n\}^K}\left(\prod_{i=1}^K\Prob(M_{n,i}=r_i)\right)
    \sum_{k=1}^n k\Prob(M_n\ge k)\Prob(Y_{r_1,\dots,r_K}=k),
\end{align*}
where $Y_{r_1,\dots,r_K}=\tau|M_{n,1}=r_1,\dots,M_{n,K}=r_K,S_1=1,\dots,S_k=k$. The idea is that conditional on the sample paths, the stopping time is independent of the process again.
Using $\Prob(M_n=k)=(1-p)^{1\{k<n\}}p^k$, $\Prob(M_n\ge k)=p^k$ for $k = 1,\dots,n$, and $\E[M_n]=\frac{p}{1-p}(1-p^n)$, the prophet ratio can be further rewritten as
\begin{align*}
    \frac{\E[S_\tau^+]}{\E[M_n]}&=\frac{1}{\E[M_n]}\sum_{\boldsymbol{r}\in\{0,\dots,n\}^K}\left(\prod_{i=1}^K\Prob(M_{n,i}=r_i)\right)
    \sum_{m=1}^n m\Prob(M_n\ge m)\Prob(Y_{r_1,\dots,r_K}=m)\\
    &=\sum_{\boldsymbol{r}\in\{0,\dots,n\}^K}\sum_{m=1}^n\Prob(Y_{r_1,\dots,r_K}=m)(1-p)^{\sum_{i=1}^K1\{r_i < n\}}p^{\sum_{i=1}^Kr_i}mp^{m-1}\frac{1-p}{1-p^n}.
\end{align*}
First taking the supremum over all distributions of $Y$ and then the infimum over all $p\in(0,1)$ yields the upper bound on the prophet constant $C_{n,K}$. Using the same linearisation technique as in Theorem \ref{thm:finite-horizon-no-information}, this upper bound can be written as the following semi-infinite LP.
{
\nobreakdisplays
\begin{align*}
    &\sup && C\\
    &\text{subject to} &&\sum_{\boldsymbol{r}\in\{0,\dots,n\}^K}\sum_{m=1}^n\Prob(Y_{r_1,\dots,r_K}=m)m(1-p)^{\sum_{i=1}^K1\{r_i < n\}}p^{m-1+\sum_{i=1}^Kr_i}\frac{1-p}{1-p^n}\ge C\\&&&\qquad\qquad\qquad\qquad\qquad\qquad\qquad\qquad\qquad\qquad\qquad\qquad\text{ for all } p\in(0,1),\\
    &&& \kappa(\boldsymbol{r})\in\mathcal{P}(\{1,\dots,n\})\text{ for all }\boldsymbol{r}\in\{0,\dots,n\}^K\text{, }C\in\mathbb{R}.
\end{align*}
}
By comparing to our general form of the semi-infinite LP, it is clear that the dual of the above problem is
    {
    \nobreakdisplays
    \begin{align*}
        &\inf && \sum_{\boldsymbol{r}\in\{0,\dots,n\}^K}\lambda(\boldsymbol{r})\\
        &\text{subject to} &&\lambda(\boldsymbol{r})\ge \int_0^1(1-p)^{\sum_{i=1}^K1\{r_i < n\}}p^{\sum_{i=1}^Kr_i}mp^{m-1}\frac{1-p}{1-p^n}\dd\mu (p)\\&&&\qquad\qquad\qquad\:\text{for all }\boldsymbol{r}\in\{0,\dots,n\}^K\text{ and }m\in\{1,\dots,n\},\\
        &&&\mu\in\mathcal{P}((0,1))\text{, }\lambda:\{0,\dots,n\}^K\to\mathbb{R},
    \end{align*}
    }
    where the measure $\mu$ is associated with the continuum of constraints.
 
    We will first provide a feasible dual solution $(\nu,\zeta)$ and then provide an upper bound on the objective. Consider the measure $\dd \nu(p)=\frac{1}{H_n}\frac{1-p^n}{1-p}\dd p$ (recall from Theorem \ref{thm:finite-horizon-no-information} that it is a probability measure  and let $\zeta(\boldsymbol{r})=\beta\left(\sum_{i=1}^Kr_i,\sum_{i=1}^K\
    \indicator{r_i < n}\right)$, where
    \begin{align*}
        \beta(s,I)&:=\sup_{m\in\{1,\dots,n\}}\int_0^1(1-p)^Ip^smp^{m-1}\frac{1-p}{1-p^n}\dd\nu (p)\\
        &=\sup_{m\in\{1,\dots,n\}} \frac{m}{H_n}\int_0^1(1-p)^Ip^{s+m-1}\dd p= \frac{1}{H_n}\sup_{m\in\{1,\dots,n\}}mB(s+m,I+1).
    \end{align*}
    By construction the pair $(\nu,\zeta)$ is a feasible solution to the dual (for similar reasons as the corresponding arguments in Theorem \ref{thm:finite-horizon-no-information}). In particular, $\beta(s,I)$ is defined so that $\zeta(\boldsymbol{r})$ satisfies the constraints.

    The objective of this dual feasible solution is difficult to calculate, so instead an upper bound will be determined.
    Note that this upper bound is still an upper bound for the primal problem by weak duality.
    Our first step is the following decomposition which results in more well-behaved sums.
    \begin{align}
        \label{eq:hypercube-decomposition}
        \sum_{\boldsymbol{r}\in\{0,\dots, n \}^K}\zeta(\boldsymbol{r})&=\sum_{\boldsymbol{r}\in\{0,\dots, n \}^K}\beta\left(\sum_{i=1}^Kr_i,\sum_{i=1}^K\
    \indicator{r_i < n}\right)\nonumber \\
    &=\beta(nK,0)+\sum_{I=1}^K\binom{K}{I}\sum_{\boldsymbol{r}\in\{0,\dots, n-1 \}^I}\beta\left(n(K-I)+\sum_{i=1}^{I}r_i,I\right) \nonumber \\
    & =\beta(nK,0)+ \sum_{\boldsymbol{r}\in\{0,\dots, n-1 \}^K}\beta\left(\sum_{i=1}^{K}r_i,K\right) \nonumber \\
    & \equalSpace+ \sum_{I=1}^{K-1}\binom{K}{I}\sum_{\boldsymbol{r}\in\{0,\dots, n-1 \}^I}\beta\left(n(K-I)+\sum_{i=1}^{I}r_i,I\right)
    \end{align}
    where the second term corresponds to the case $I = K$ in the first summation.
    The first term is
    \begin{equation}
        \label{eq:term-one}
        \beta(nK,0)=\sup_{m\in\{1,\dots,n\}}\frac{m}{H_n}B(nK+m,1)=\sup_{m\in\{1,\dots,n\}}\frac{1}{H_n}\frac{m}{nK+m}=\frac{1}{H_n}\frac{1}{K+1}.
    \end{equation}
    The second term, the case $I=K$, is covered by Lemma \ref{lemma:interior-objective} which shows that
    \begin{equation}
        \label{eq:term-two}
        \sum_{\boldsymbol{r}\in\{0,\dots, n-1 \}^K}\beta\left(\sum_{i=1}^{K}r_i,K\right)<\left(\frac{K}{K+1}\right)^{K+1}+\frac{3+3H_K}{H_n}.
    \end{equation}
    Using that $\beta(s,K)\le\frac{1}{H_n}\gamma(s,K)$, where $\gamma$ is as in the statement of Lemma \ref{lemma:interior-objective}.
    What remains is 
    $$
    \sum_{I=1}^{K-1}\binom{K}{I}\sum_{\boldsymbol{r}\in\{0,\dots, n-1 \}^I}\beta\left(n(K-I)+\sum_{i=1}^{I}r_i,I\right).
    $$
    Let $U_i\iid\text{Unif}(0,1)$ and note that $\lfloor n U_i\rfloor$ has a uniform distribution over $\{0,\dots,n-1\}$. Using this, 
    \begin{align*}
        \sum_{\boldsymbol{r}\in\{0,\dots, n-1 \}^I}\beta\left(n(K-I)+\sum_{i=1}^{I}r_i,I\right)&< \frac{(I-1)!}{eH_n}\sum_{\boldsymbol{r}\in\{0,\dots, n-1 \}^I}\left(n(K-I)+\sum_{i=1}^{I}r_i\right)^{-I}\\
        &= \frac{(I-1)!}{eH_n}\E\left[\left(K-I+\frac{1}{n}\sum_{i=1}^I\lfloor nU_i \rfloor\right)^{-I}\right]\\
        &\le \frac{(I-1)!}{eH_n}\E\left[\left(K-I+\sum_{i=1}^IU_i -\frac{I}{n}\right)^{-I}\right].
    \end{align*}
    In the first step we use the bound from Lemma \ref{lemma:constraint-bound} and $(I/(I+1))^{I+1}\le 1/e$. In the second step that $$\left(n(K-I)+\sum_{i=1}^{I}r_i\right)^{-I} = n^{-I}\left(K-I+\frac{1}{n}\sum_{i=1}^{I}r_i\right)^{-I}$$ and that $n^{-I}$ is the probability of seeing a certain $\boldsymbol{r} \in \{0,\dots,n-1\}^I$ under the uniform distribution over this set. Finally, in the third inequality that $\frac{1}{n}\lfloor nU_i \rfloor\ge U_i-\frac{1}{n}$
    We now make use of $n\ge 2K^2$ to obtain
    \begin{align*}
        \left(K-I+\sum_{i=1}^IU_i -\frac{I}{n}\right)^{-I}&=\left(1-\frac{I/n}{K-I+\sum_{i=1}^IU_i}\right)^{-I}\left(K-I+\sum_{i=1}^IU_i\right)^{-I}\\
        &\le \left(1-I/n\right)^{-I}\left(K-I+\sum_{i=1}^IU_i\right)^{-I}\\
        &\le (1+4I^2/n)\left(K-I+\sum_{i=1}^IU_i\right)^{-I}\qquad\qquad\text{(Lemma \ref{lemma:exponent-upper-bound})}\\
        &\le (1+4K^2/n)\left(K-I+\sum_{i=1}^IU_i\right)^{-I}.
    \end{align*}
    Using for the first inequality $K-I + \sum_{i=1}^I U_i \geq 1$.
    Because of the above bound, we can work with $U_i$ instead of $\frac{1}{n}\lfloor nU_i \rfloor$. 
    The key identity to continue is a rearranged version of \eqref{eq:gamma-integral}
    $$
    x^{-I}=\frac{1}{(I-1)!}\int_0^\infty t^{I-1}e^{-xt}\dd t.
    $$
    Plugging in $x = (K - I + \sum_{i=1}^I U_i)$ and taking the expectation then gives
    \begin{align*}
        \frac{(I-1)!}{eH_n}\E\left[\left(K-I+\sum_{i=1}^IU_i\right)^{-I}\right]&=\frac{1}{eH_n}\E\left[\int_0^\infty t^{I-1}e^{-t\left(K-I+\sum_{i=1}^IU_i\right)}\dd t\right]\\
        &=\frac{1}{eH_n}\int_0^\infty t^{I-1}e^{-t\left(K-I\right)}\E\left[e^{-t\sum_{i=1}^IU_i}\right]\dd t\\
        &=\frac{1}{eH_n}\int_0^\infty t^{I-1}e^{-t\left(K-I\right)}\left(\frac{1-e^{-t}}{t}\right)^{I}\dd t\\
        &=\frac{1}{eH_n}\int_0^\infty \frac{(e^t-1)^Ie^{-tK}}{t} \dd t,
    \end{align*}
    using the moment generating function of the uniform distribution in the third step, i.e., that $\mathbb{E}[e^{-tU}] = (1-e^{-t})/t$ if $U$ is a uniform random variable on $[0,1]$, and the independence of the $U_i$ for $i = 1,\dots,I$.
    Finally, 
    \begin{align}
        \label{eq:term-three}
        &\equalSpace\sum_{I=1}^{K-1}\binom{K}{I}\sum_{\boldsymbol{r}\in\{0,\dots, n-1 \}^I}\beta\left(n(K-I)+\sum_{i=1}^{I}r_i,I\right)\nonumber\\
        &\le \frac{1+4K^2/n}{eH_n}\sum_{I=1}^{K-1}\binom{K}{I}\int_0^\infty \frac{(e^t-1)^Ie^{-tK}}{t} \dd t\nonumber\\
        &=\frac{1+4K^2/n}{eH_n}\int_0^\infty\sum_{I=1}^{K-1}\binom{K}{I} \frac{(e^t-1)^Ie^{-tK}}{t} \dd t\nonumber\\
        &=\frac{1+4K^2/n}{eH_n}\int_{0}^{\infty}\frac{1-e^{-tK}-\left(1-e^{-t}\right)^{K}}{t}dt\nonumber\\
        &\le \frac{1+4K^2/n}{eH_n}\left(\int_{0}^{1/K}Kdt+\int_{1/K}^{K}\frac{1}{t}dt+\int_{K}^{\infty}\frac{1-\left(1-e^{-t}\right)^{K}}{K}dt\right)\nonumber\\
        &\le \frac{1+4K^2/n}{eH_n}\left(1+2\log K+\int_{K}^{\infty}e^{-t}dt\right)\nonumber\\
        &\le \frac{1+4K^2/n}{eH_n}\left(2+2\log K\right).
    \end{align}
    Using the binomial theorem in the third step, using in the fourth step the inequalities
    $$
    \frac{1-e^{-tK}-\left(1-e^{-t}\right)^{K}}{t}\le K \text{ and }1-e^{-tK}-\left(1-e^{-t}\right)^{K}\le 1-\left(1-e^{-t}\right)^{K}\le 1,
    $$
    and using in the fifth step $1-(1-x)^K\le Kx$.
    Adding \eqref{eq:term-one}, \eqref{eq:term-two} and \eqref{eq:term-three} together gives the final bound on the objective.
    $$
    \left(\frac{K}{K+1}\right)^{K+1}+\frac{1}{H_n}\left(\frac{1}{K+1}+3+3H_K+\frac{1}{e}(1+4K^2/n)\left(2+2\log K\right)\right)
    $$
    The bound can be made cleaner with a bit of loss by invoking $n\ge 2K^2$ and $H_K\ge \log K$
    \begin{equation*}
        \left(\frac{K}{K+1}\right)^{K+1}+\frac{6+6H_K}{H_n}.\qedhere
    \end{equation*}
    \end{proof}
\newpage

\section{Finite-horizon algorithmic result}
\label{appendix:finite-horizon-prophet}
The following result was obtained after the original submission of this paper. It resolves the special case left open in Appendix \ref{appendix:finite-horizon-unit-random-walk}. 

The initial version of the argument was provided by an AI-system.
This argument, however, closely resembles \citet{samuel1984comparison} her argument for establishing a prophet inequality with the median as threshold. 
This special case was subsequently generalised to multiple samples and Wittmann's setting by the authors. The proof below is the authors’ formulation.
\begin{theorem}
    Let $(Y_i)_{1\le i \le n}$ be a sequence of non-negative random variables 
    and let $C \ge 0$ be such that, for every $1\le i\le n$,
    $$
        \E\left[\max_{i\le j\le n}(Y_j-Y_i)|\mathcal{F}_i\right]\le C\cdot\E\left[M\right],
    $$
    where $M=\max_{1\le j\le n}Y_j$ and $\mathcal{F}_i$ is the $\sigma$-algebra generated by $Y_1,\dots,Y_i$.
    Then, given access to $K$ independent samples from the distribution of $M$, there exists a stopping rule $\tau$ such that
    $$\E\left[Y_{\tau}\indicator{\tau\le n}\right]\ge
    \beta_C(K)\E[M]
    $$
    with $\beta_C(K)=\sup\{x-Cx^2:x \in A_K\}$ and $A_K=\{a/b:a,b\in\mathbb{N}\text{ and }0< a<b\le K+1\}$.
\end{theorem}
\begin{proof}
    For clarity we assume that the distribution of $M$ is continuous. The below argument can easily be extended to the non-continuous case by adding a uniformly random tiebreaker. 
    
    Denote by $\tau(\alpha)=\inf\{1\le i\le n: Y_i > \alpha\}$ the stopping time associated with the threshold $\alpha$. On the event $\{\tau(\alpha)\le n\}$ we have that $Y_{\tau(\alpha)}>\alpha\ge Y_i$ for $i<\tau(\alpha)$, so on this event 
    $$
        M=\max_{\tau(\alpha)\le j\le n}Y_j=Y_{\tau(\alpha)}+\max_{\tau(\alpha)\le j\le n}(Y_j-Y_{\tau(\alpha)}).
    $$
    The above is a standard decomposition. Taking expectations yields
        \begin{align*}
        \E[M\indicator{\tau(\alpha)\le n}]&=\E[Y_{\tau(\alpha)}\indicator{\tau(\alpha)\le n}]+\sum_{i=1}^n\E\left[\indicator{\tau(\alpha) = i}\max_{i\le j\le n}(Y_j-Y_i)\right] \\
        &=\E[Y_{\tau(\alpha)}\indicator{\tau(\alpha)\le n}]+\sum_{i=1}^n\E\left[\indicator{\tau(\alpha) = i}\E\left[\max_{i\le j\le n}(Y_j-Y_i)|\mathcal{F}_i\right] \right]\\
        &\le\E[Y_{\tau(\alpha)}\indicator{\tau(\alpha)\le n}]+\sum_{i=1}^n\E\left[\indicator{\tau(\alpha) = i}C\E[M]\right]\\
        &=\E[Y_{\tau(\alpha)}\indicator{\tau(\alpha)\le n}]+C\E[M]\Prob(\tau(\alpha)\le n).
    \end{align*}
    Using in the second step that $\{\tau(\alpha)=i\}$ is $\mathcal{F}_i$-measurable. Denote by $M_{(i)}$ the $i$-th order statistic of the $K$ auxiliary samples. For fixed $1\le j\le K$ let $T=M_{(j)}$ be the threshold that is used. Let $N_i=M_i$, $N_{K+1}=M$ and let $N_{(i)}$ denote the joint $i$-th order statistic.
    By exchangeability (and no ties due to the continuity assumption) we find that
    $$
    \Prob(\tau(T)\le n)=\Prob(M\ge T)=\frac{K+1-j}{K+1}.
    $$
    The event $\{\tau(T)>n\}$ implies that $M$ is (uniformly) an order statistic $N_{(i)}$ with $1\le i\le j$, so
    $$
    \E[M\indicator{\tau > n}]=\sum_{i=1}^{j}\frac{1}{K+1}\E[N_{(i)}]\le\frac{j}{K+1}\E[N_{(j)}].
    $$
    On the event $\{\tau(T)\ge n\}$, $T=N_{(j)}$, so 
    $$
    \E[Y_{\tau(T)}\indicator{\tau(T)\le n}]\ge\E[T\indicator{\tau(T)\le n}] =\frac{K+1-j}{K+1}\E[N_{(j)}].
    $$
    All three taken together yield
    \begin{align*}
        \E[M]&=\E[M\indicator{\tau(T)\le n}]+\E[M\indicator{\tau(T)> n}]\\
        &\le \E[Y_{\tau(T)}\indicator{\tau(T)\le n}]+C\E[M]\frac{K+1-j}{K+1}+\frac{j}{K+1}\E[N_{(j)}]\\
        &\le \left(1+\frac{j}{K+1-j}\right)\E[Y_{\tau(T)}\indicator{\tau(T)\le n}]+C\E[M]\frac{K+1-j}{K+1}.
    \end{align*}
    Rearranging the inequality yields
    $$
    \E[Y_{\tau(T)}\indicator{\tau(T)\le n}]\ge \left(\left(\frac{K+1-j}{K+1}\right)-C\left(\frac{K+1-j}{K+1}\right)^{2}\right)\E[M].
    $$
    The optimal constant is equal to $\sup\{p-Cp^2:p\in\{1/(K+1),\dots,K/(K+1)\}\}$. Observe, however, that samples can also be discarded, so we recover that the constant is $\beta_C(K)$ (the supremum can be taken over previous sets).
\end{proof}
As noted by \citet{wittmann1995prophet}, the case $C=1$ corresponds to the random walk setting. For $C=1$ we find a prophet constant of $1/4$ for $K\ge 1$. For $C=2$ we find no prophet inequality for $K=1$, a $1/9$ for $K=2$ and a constant of $1/8$ for $K\ge 3$. 
More generally, 
$$
\lim_{K\to\infty}\beta_C(K)=\sup\{p-Cp^2:p\in[0,1]\}.
$$
When $C$ is rational this limit is attained for finite $K$ and otherwise it converges to it quickly. 
This limiting constant is also identified in \citet[Theorem 7]{wittmann1995prophet}.

To achieve more, like Wittmann does in his first theorem, the values of the samples have to be used in a more sophisticated manner. We leave improving the above theorem open, but remark that, as noted by Wittmann, that for $C=1$ the best possible prophet constant with full information is at most $1/2.881$ and that to achieve $1/e$ in the many-samples limit requires more exploitation of the random walk structure.

For posterity, the original partial result is preserved below.
\subsection{\texorpdfstring{$\pm1$}{+-1} finite-horizon algorithmic result}
\label{appendix:finite-horizon-unit-random-walk}
The existing literature on the full-information setting for the finite-horizon case does not readily extend to our sample-based information setting.
The threshold-based approach of \citet{wittmann1995prophet}, for instance, relies on precise distributional information, and replacing these thresholds by natural sample-based estimates leads to poor control of overshoots. 
Likewise, the framework of \citet{immorlica2023prophet} appears promising, but does not apply in our setting since it assumes non-negative increments. 
Moreover, as already remarked, the finite-horizon setting lacks the ladder height decomposition used by our infinite-horizon setting results.

By imposing an assumption, however, we do recover a ladder height like decomposition. Specifically, under the assumption that the increments are $\pm 1$,
the running maximum of the random walk can only increase with steps of $+1$, i.e., the ladder height increments $J_i$ are $+1$. This eliminates the dependence between the ladder height increments $J_i$ and intervals $I_i$ (because the former is constant), allowing for a more simplified analysis.

Using this simplification, we can establish the following (non-optimal) performance inequality.
\begin{theorem}
    \label{thm:partial-stopping-K-samples-special-case} 
    Let the distribution $F$ have support in $\{-1,+1\}$, let $(S_k)_{0\le k\le n}$ be a random walk with increments $X_i\iid F$ and let $(S'_k)_{0\le k\le n}$ be an independent copy. 
    Then there is a stopping rule $\stoppingRule$ that has access to auxiliary data $D=((S'_k)^+)_{0 \leq k \leq n}$, but no knowledge of $F$, such that
    $$
    \E[S_{\stoppingTime}^+]\ge0.1208\cdot \E\left[\max_{0\le k\le n}S_k\right].
    $$
\end{theorem}
To prove this result, first two structural lemmas.
\begin{lemma}
\label{lemma:NBU}
    Let $I_i$ be non-negative i.i.d. random variables and let $G_n=\max\{k:\sum_{i=1}^kI_i\le n\}$, then for non-negative integers $a$ and $b$ $$\mathbb{P}(G_n\ge a+b)\le\mathbb{P}(G_n\ge a)\mathbb{P}(G_n\ge b).$$
\end{lemma}
\begin{proof}
    This result is well-known in renewal theory, we give the proof here for completeness. Let $\tilde{G_n}_n=\min\{k:\sum_{i=1}^kI_i>n\}$, we have that 
    \begin{align*}
        \mathbb{P}(\tilde{G}_n>a+b)&=\mathbb{P}\left(\sum_{i=1}^1I_i\le n,\dots,\sum_{i=1}^{a+b}I_i\le n\right)\\
        &\le\mathbb{P}\left(\sum_{i=1}^1I_i\le n,\dots,\sum_{i=1}^{a}I_i\le n,\sum_{i=a+1}^{a+1}I_i\le n\dots,\sum_{i=a+1}^{a+b}I_i\le n\right)\\
        &=\mathbb{P}\left(\sum_{i=1}^1I_i\le n,\dots,\sum_{i=1}^{a}I_i\le n\right)\mathbb{P}\left(\sum_{i=a+1}^{a+1}I_i\le n\dots,\sum_{i=a+1}^{a+b}I_i\le n\right)\\
        &=\mathbb{P}(\tilde{G}_n>a)\mathbb{P}(\tilde{G}_n>b). 
    \end{align*}
    The penultimate equality follows from independence.
    Observe that $G_n+1=\tilde{G}_n$, therefore 
    $\Prob(\tilde{G}_n>a)=\Prob(G_n+1>a)\Prob(G_n\ge a)$, which gives the result. 
\end{proof}
We continue with the second lemma, building on the condition shown in the first lemma.
\begin{lemma}
    \label{lemma:bernoulli-rw-rewrite}
    Let $G_n$ and $G_n'$ be  i.i.d. non-negative integer-valued random variables that satisfy $\mathbb{P}(G_n\ge a+b)\le \mathbb{P}(G_n\ge a)\mathbb{P}(G_n\ge b)$ for all non-negative integers $a$ and $b$. Then
    \begin{align*}
    \E[\max\{G_n',1\}\indicator{G_n\ge \max\{G_n',1\}}]&\ge\frac{1}{4}\E[G_n]+\left(\frac{3}{4}-\Prob(G_n\ge 1)\right)\Prob(G_n\ge 1)\\&+\,\frac{1}{2}\sum_{g=1}^\infty g\Prob(G_n=g)^2.
    \end{align*}
\end{lemma}
\begin{proof}
    First of all, by resolving the case $G_n'=0$ we find that
    $$
    \E[\max\{G_n',1\}\indicator{G_n\ge \max\{G_n',1\}}]=\mathbb{P}(G_n'=0)\mathbb{P}(G_n\ge 1)+\E[G_n'\indicator{G_n\ge G_n'}].
    $$
    The result now follows from rewriting
    \begin{align*}
        \E[G_n'\indicator{G_n\ge G_n'}]&=\frac{1}{2}\E[G_n'\indicator{G_n\ge G_n'}+G_n\indicator{G_n'< G_n}+G\indicator{G_n'= G_n}]\\
            &=\frac{1}{2}\E[\min\{G_n,G_n'\}]+\frac{1}{2}\sum_{g=1}^\infty g\mathbb{P}(G_n=g)^2\\
            &=\frac{1}{2}\sum_{g=1}^\infty\mathbb{P}(G_n\ge g)^2+\frac{1}{2}\sum_{g=1}^\infty g\mathbb{P}(G_n=g)^2\\
            &\ge\frac{1}{2}\sum_{g=1}^\infty\mathbb{P}(G_n\ge 2g)+\frac{1}{2}\sum_{g=1}^\infty g\mathbb{P}(G_n=g)^2\\
            &=\frac{1}{4}\sum_{g=1}^\infty(\mathbb{P}(G_n\ge 2g)+\mathbb{P}(G_n\ge 2g+1))+\frac{1}{2}\sum_{g=1}^\infty g\mathbb{P}(G_n=g)^2\\
            &\ge\frac{1}{4}\E[G_n]-\frac{1}{4}\mathbb{P}(G_n\ge 1)+\frac{1}{2}\sum_{g=1}^\infty g\mathbb{P}(G_n=g)^2.
    \end{align*}
    Adding the expressions together gives the result.
\end{proof}

With the above lemma the following inequality can be established.

\begin{lemma}
\label{lemma:submult-expectation}
    Let $G_n$ and $G_n'$ be  i.i.d. non-negative integer-valued random variables that satisfy $\mathbb{P}(G_n\ge a+b)\le \mathbb{P}(G_n\ge a)\mathbb{P}(G_n\ge b)$ for all non-negative integers $a$ and $b$, then for $\alpha=2-\frac{1}{2}\sqrt{14}$
    \begin{align*}
    \E[\max\{G_n',1\}\indicator{G_n\ge \max\{G_n',1\}}]&\ge\left(\frac{1}{4}-\alpha\right)\E[G_n].
    \end{align*}
\end{lemma}
\begin{proof}
    We continue from Lemma \ref{lemma:bernoulli-rw-rewrite}. Let $Q=\E[\max\{G_n',1\}\indicator{G_n\ge \max\{G_n',1\}}]$ and let $p_i=\mathbb{P}(G_n=0)$. Observe that $Q\ge (1/4-\alpha)E[G_n]$ if 
    $$
    \alpha\E[G_n]+\left(p_0-\frac{1}{4}\right)(1-p_0)+\frac{1}{2}\sum_{g=1}^\infty gp_g^2\ge 0.
    $$
    Using $\E[G_n]\ge 2 - 2p_0-p_1$ we obtain the sufficient condition
    $$
    \alpha\E[G_n]+\left(p_0-\frac{1}{4}\right)(1-p_0)+\frac{1}{2}\sum_{g=1}^\infty gp_g^2\ge\alpha(2-2p_0-p_1)+\left(p_0-\frac{1}{4}\right)(1-p_0)+\frac{1}{2}p_1^2\ge 0.
    $$
    Write 
    \begin{align*}
        f(\alpha,p_0,p_1)&=\alpha(2-2p_0-p_1)+\left(p_0-\frac{1}{4}\right)(1-p_0)+\frac{1}{2}p_1^2\\
        &=\frac{1}{2}p_1^2-p_0^2+\frac{5}{4}p_0+\alpha(2-2p_0-p_1)-\frac{1}{4}.
    \end{align*}
    The map $p_0\mapsto f(\alpha, p_0,p_1)$ is concave and attains a minimum at its endpoints, so
    \begin{align*}
        &\equalSpace\min\{f(\alpha,p_0,p_1):p_0\ge 0, p_1\ge 0, p_0+p_1\le 1\}\\
        &=\min\{\min_{0\le p_1\le 1}\{f(\alpha, 0, p_1)\}, \min_{0\le p_1\le 1}\{f(\alpha, 1-p_1, p_1)\}\}
    \end{align*}
    For $\alpha=2-\frac{1}{2}\sqrt{14}$, the map $p_1\mapsto f(\alpha, 0, p_1)$ is non-negative, and for the same $\alpha$ the map $p_1\mapsto f(\alpha, 1-p_1, p_1)$ is a concave, and attains its minimum at its endpoints, these endpoints evaluate to $f(\alpha,1, 0)=0$ and $f(\alpha,0,1)>0$. We therefore have that the above minimisation problem is non-negative for $\alpha=2-\frac{1}{2}\sqrt{14}$, and consequently that the desired inequality holds for $\alpha=2-\frac{1}{2}\sqrt{14}$.
\end{proof}

The proof is now straightforward. 
\begin{cleverproof}{thm:partial-stopping-K-samples-special-case} 
    For a random walk with $\pm 1$ increments the finite-horizon maximum is the amount of ladder heights, i.e., $\max_{0\le k\le n}S_k=G_n$.
    
    From the auxiliary data an i.i.d. copy $G_n'$ of $G_n$ (the amount of ladder heights in the process) can be constructed by counting. A stopping rule that stops at the $\max\{1,G_n'\}$-th ladder height obtains a payoff of at least 
    $$
    \left(\frac{1}{2}\sqrt{14}-\frac{7}{4}\right)\E[G_n]
    $$
    by the above lemma. This therefore results in a prophet constant of at least $\frac{1}{2}\sqrt{14}-\frac{7}{4}>0.1208$. 
\end{cleverproof}
This prophet constant can be improved with more refined arguments, but we leave this as an open problem. Moreover, note that this result holds when $I_i$ and $J_i$ are independent. This also occurs when the support is a subset of $\{z\in\mathbb{Z}:z\le 1\}$, as then the ladder heights increments can also only be $+1$.

\end{document}